\documentclass[amssymb,twoside,11pt]{article}
\usepackage{latexsym,amsmath,graphicx}
\usepackage{amsfonts}
\newtheorem{theorem}{Theorem}[section]
\newtheorem{lemma}[theorem]{{\bf Lemma}}

\newtheorem{rem}[theorem]{{\bf Remark}}

\newtheorem{definition}{Definition}[section]
\numberwithin{equation}{section}
\newenvironment{proof}{\indent{\em Proof:}}{\quad \hfill
$\Box$\vspace*{2ex}}

\font\Bbb=msbm10 at 12pt

\newcommand{\R}{\mbox{\Bbb R}}
\newcommand{\N}{\mbox{\Bbb N}}
\newcommand{\C}{\mbox{\Bbb C}}
\begin{document}
\setcounter{page}{1}
\begin{center}
\vspace{0.4cm} {\large{\bf On the Nonlinear $\Psi$-Hilfer Fractional  Differential Equations}} \\
\vspace{0.5cm}
Kishor D. Kucche $^{1}$ \\
kdkucche@gmail.com \\

\vspace{0.35cm}
Ashwini D. Mali  $^{2}$\\
maliashwini144@gmail.com\\

\vspace{0.5cm}
J. Vanterler da C. Sousa$^{3}$ \\
ra160908@ime.unicamp.br \\

\vspace{0.35cm}
$^{1,2}$ Department of Mathematics, Shivaji University, Kolhapur-416 004, Maharashtra, India.\\
$^3$ Department of Applied Mathematics, Imecc-Unicamp,
13083-859, Campinas, SP, Brazil.
\end{center}

\def\baselinestretch{1.0}\small\normalsize

\begin{abstract}
We consider the nonlinear Cauchy problem for $ \Psi $- Hilfer fractional differential equations and investigate the existence, interval of existence and uniqueness  of solution in the weighted space of functions. The continuous dependence of solutions on initial conditions is proved via Weissinger fixed point theorem. Picard’s successive approximation method has been developed to solve nonlinear Cauchy problem for differential equations with $ \Psi $- Hilfer fractional derivative and an estimation have been obtained for the error bound.
Further, by Picard's successive approximation, we derive the representation formula for the solution of linear Cauchy problem for $ \Psi $-Hilfer fractional differential equation with constant coefficient and variable coefficient in terms of Mittag-Leffler function and Generalized (Kilbas--Saigo) Mittag-Leffler function.

\end{abstract}
\def\baselinestretch{1.5}
\allowdisplaybreaks
\section{Introduction}
During the last decade, fractional differential equations(FDEs) \cite{Diethelm,Kilbas,Podlubny} appeared as rich and beautiful field of research due to their applications to the physical and life sciences.The theoretical development of FDEs in the classical Riemann-Liouville's or Caputo sense have been excellently given in \cite{Diethelm,Kilbas,Podlubny,Gorenflo,Miller,Lak3}. For important and interesting works on existence, uniqueness, data dependence and  heterogeneous qualitative properties of solution of  classical FDEs, we propose the work of  Lakshmikantham et al.\cite{Lak3,Lak1,Lak2}, Gejji et al.\cite{Gejji1,Gejji2}, Kai et al.\cite{Diethelm1}, Trujillo et al.\cite{Kucche}, Benchora et al.\cite{Agrawal2}, and the references cited therein. Other interesting works on these aspect can be found in \cite{Agrawal2, Agrawal,Lak4,Wang,Kou,Idczak,Deng,Li,Kilbas2,Gambo,Zhou,Agrawal3,Ahmad,Anber}.

The fundamental concept and  properties of  integrals and  derivatives of fractional order of a function with respect to the another function viz. $ \Psi $-Riemann-Liouville integral and derivative have been introduced in 
\cite[Chapter 2]{Kilbas}. Following the similar approach, Almeida \cite{Almeida} introduced  $\Psi$-Caputo fractional derivative and investigated the interesting properties of this operator and extended 
few preceding study concerned with the Caputo and the Caputo--Hadamard derivative operators. 

On the other hand, Hilfer \cite{Hilfer} introduced
a fractional derivative $\mathcal{D}_{a+}^{\eta ,\nu }(\cdot)$ having two parameters  $\eta \in (n-1, n), n \in \mathbb{N}$  and  $\nu ~(0\leq \nu \leq 1)$ which in specific gives the Riemann--Liouville and the Caputo derivative operator.
Furati et al. \cite{Furati} analyzed nonlinear FDEs with  Hilfer derivative operator. Sousa and Oliveira \cite{Vanterler1,Vanterler2} introduced a new definition of fractional derivative viz. $\Psi$--Hilfer fractional derivative and  investigated its important properties. Further, it is proved that the $\Psi$--Hilfer derivative is the  generalization of many existing fractional derivative operators. In \cite{Vanterler2}, Sousa and Oliveira established  generalized Gronwall inequality through the fractional integral with respect to another function and studied the existence, uniqueness and data dependence of solution of Cauchy problem with $\Psi$--Hilfer fractional derivative. For the study relating  to existence, uniqueness and stability of different sorts of FDEs involving $\psi$-Hilfer fractional derivative operator, we refer to the work of Sousa et al.\cite{jose1,jose2,jose3} and the references given there in.

Motivated by the work of \cite{Diethelm1,Vanterler1,Vanterler2} in this paper, existence along with the interval of existence, uniqueness and continuous dependence of solutions on initial condition have examined for the nonlinear  $\Psi$-Hilfer FDEs structured as:
\begin{align}
^H \mathcal{D}^{\eta,\,\nu\,;\, \Psi}_{a +}y(t)&= f(t, y(t)),~~ ~t \in  \Lambda =[a,a+\xi], ~~\xi>0, ~0<\eta<1, ~0\leq\nu\leq 1,~\label{eq1}\\
\mathcal{I}_{a +}^{1-\zeta\,;\, \Psi}y(a)&=y_a \in \R, \qquad\zeta =\eta +\nu \left( 1-\eta \right), \label{eq2}
\end{align}
where $^H \mathcal{D}^{\eta,\nu;\, \Psi}_{a +}(\cdot)$ is the $\Psi$-Hilfer derivative of order $\eta$ and type $\nu$, $\mathcal{I}_{a +}^{1-\zeta;\, \Psi}(\cdot)$ is $\Psi$-Riemann--Liouville integral of order $1-\zeta$ and $f: \Lambda \times \R  \to \R $ is an appropriate function.

Further, by Picard’s successive approximation method we derive the existence along with uniqueness of solution for Cauchy type problem \eqref{eq1}-\eqref{eq2}. We have established bound for the error between approximated solution $y_n$ and exact solution $y$ and proved that the difference is approaches to zero when  $n$ is very large. The development of Picard’s successive approximation technique is utilized to obtain representation formula  for the solution of linear Cauchy problem with constant coefficient
 \begin{align} 
  {}^H\mathcal{D}_{a+}^{\eta,\,\nu\,; \,\Psi}y(t)-\lambda y(t)&=f( t ),~\lambda \in \R, ~0<\eta<1, ~ 0\leq\nu\leq 1,~t\in \Delta=[a,b].\label{a1} \\
\mathcal{I}_{a+}^{1-\zeta\,; \,\Psi}y(a)&=y_a \in \R, ~~\zeta  =\eta+\nu(1-\eta) \label{a2}
 \end{align}
 in form of Mittag--Leffler function, where as  the representation formula  for the solution of linear Cauchy problem with variable coefficient  
\begin{align}
{}^H\mathcal{D}_{a+}^{\eta,\,\nu\,; \,\Psi}y(t)-\lambda[\Psi(t)-\Psi(a)]^{\mu-1} y(t) &=0,~0<\eta<1, ~0\leq\nu\leq 1,~~\mu>1-\eta ,~
 t\in\Delta,\label{a3}\\
\mathcal{I}_{a+}^{1-\zeta\,; \,\Psi}y(a) &=y_a \in \R, ~~\zeta=\eta+\nu(1-\eta),\label{a4}
  \end{align}
is obtained in form of generalized (Kilbas--Saigo) Mittag--Leffler function.

This paper is divided in five sections: In the section 2, we provide some definitions, theorems of $\Psi$-Hilfer fractional derivative and the results which will be utilized throughout this paper. In Section 3, we establish the results pertaining to existence, uniqueness and continuous dependence of solution of  \eqref{eq1}-\eqref{eq2}. Section 4, discuss the convergence of Picard's type successive approximations to the solution of \eqref{eq1}-\eqref{eq2}. In Section 5, representation formulas have been obtained for the solution of linear Cauchy
problem with constant coefficient and variable coefficient.   

\section{Preliminaries} \label{preliminaries}

We review a few definitions, notations and results of $\Psi$-Hilfer fractional derivative \cite{Vanterler1, Vanterler2}. 
Let $\Delta=[a,b]$ $(0<a<b<\infty)$ be a finite interval. Consider the space  $C_{1-\zeta;\,\Psi}(\Delta,\,\R)$ of weighted functions $h$ defined on $\Delta$ given by
\begin{equation*}
C_{1-\zeta ;\, \Psi }(\Delta,\,\R) =\left\{ h:\left( a,b\right]
\rightarrow \mathbb{R}~\big|~\left( \Psi \left( t\right) -\Psi \left(
a\right) \right) ^{1-\zeta }h\left( t\right) \in C(\Delta,\,\R)
\right\} ,\text{ }0< \zeta \leq 1
\end{equation*}
endowed with the norm
\begin{equation}\label{space1}
\left\Vert h\right\Vert _{C_{1-\zeta ;\,\Psi }\left(\Delta,\R\right)  }=\underset{t\in \Delta 
}{\max }\left\vert \left( \Psi \left( t\right) -\Psi \left( a\right) \right)
^{1-\zeta }h\left( t\right) \right\vert.
\end{equation}
\begin{definition} 

Let $\eta>0 ~(\eta \in \R)$, $h \in L_1(\Delta,\,\R)$ and $\Psi\in C^{1}(\Delta,\,\R)$ be an increasing function wth $\Psi'(t)\neq 0$, for all $\, t\in \Delta$. Then, the $\Psi$-Riemann--Liouville fractional integral of a function $h$ with respect to $\Psi$ is defined by 
\end{definition}
\begin{equation}\label{P1}
\mathcal{I}_{a+}^{\eta \, ;\,\Psi }h\left( t\right) =\frac{1}{\Gamma \left( \eta
\right) }\int_{a}^{t}\mathrm{Q}^{\eta}_{\Psi}(t,\sigma)h(\sigma) \,d\sigma.
\end{equation}
where,$\mathrm{Q}^{\eta}_{\Psi}(t,s)=\Psi'(s)(\Psi(t)-\Psi(s))^{\eta-1},~ t,s\in \Delta.$

\begin{definition}  Let $n-1<\eta <n \in \N $ and $h ,\Psi\in C^{n}(\Delta,\,\R)$ two functions such that $\Psi$ is increasing with $\Psi'(t)\neq 0,$ for all $\,t\in \Delta$. Then, the  $( \mbox{left-sided} )$ $\Psi$-Hilfer fractional derivative $^{H}\mathcal{D}^{\eta,\,\nu\, ;\,\Psi}_{a+}(\cdot)$  of a function $h$ of order $\eta$ and type $0\leq \nu \leq 1$, is defined by
\begin{equation}\label{HIL}
^{H}\mathcal{D}_{a+}^{\eta ,\,\nu \, ;\,\Psi }h\left(t\right) =\mathcal{I}_{a+}^{\nu \left(
n-\eta \right) \, ;\,\Psi }\left( \frac{1}{\Psi ^{\prime }\left( t\right) }\frac{d}{dt}\right) ^{n}\mathcal{I}_{a+}^{\left( 1-\nu \right) \left( n-\eta
\right) \, ;\,\Psi }h\left( t\right).
\end{equation}
\end{definition}

Following results from \cite{Vanterler1, Vanterler2} play key role in proving our main results.

\begin{lemma}\label{lema1} If $\eta >0$ and \, $0\leq \mu <1,$ then $\mathcal{I}_{a+}^{\eta \, ;\,\Psi }(\cdot)$ is bounded from $C_{\mu \, ;\,\Psi }\left[ a,b\right] $ to $C_{\mu \, ;\,\Psi }\left[ a,b\right] .$ In addition, if $\mu \leq \eta $, then $\mathcal{I}_{a+}^{\eta \, ;\,\Psi }(\cdot)$ is bounded from $C_{\mu \, ;\,\Psi }\left[ a,b\right] $ to $C\left[ a,b\right] $.
\end{lemma}

\begin{lemma}\label{lema2} Let $\eta>0$ and $\delta>0$. If $h(t)= \left( \Psi \left( t\right) -\Psi \left( a\right)
\right) ^{\delta -1}$, then 
\begin{equation*}
\mathcal{I}_{a+}^{\eta \, ;\,\Psi }h(t)=\frac{\Gamma \left( \delta \right) }{\Gamma \left(
\eta +\delta \right) }\left( \Psi \left(t\right) -\Psi \left( a\right)
\right) ^{\eta +\delta -1}.
\end{equation*}

\end{lemma}

\begin{lemma}\label{lema3} Let $\Psi\in C^{1}(\Delta,\mathbb{R})$ be  increasing function with $\Psi'(t)\neq 0$, for all $t\in\Delta$. If $\zeta =\eta +\nu \left( 1-\eta \right) $ where, $ 0<\eta <1$ and $0\leq \nu \leq 1,$ then $\Psi$-Riemann-Liouville fractional integral operator $\mathcal{I}_{a+}^{\eta \, ;\,\Psi }\left( \cdot \right): C_{1-\zeta \, ;\,\Psi }\left[ a,b\right]  \rightarrow C_{1-\zeta \, ;\,\Psi }\left[
a,b\right]$ is bounded and it is given by :
\begin{equation}\label{eq23}
\left\Vert \mathcal{I}_{a+}^{\eta \, ;\,\Psi }h\right\Vert _{C_{1-\zeta \, ;\,\Psi }\left[ a,b \right] }\leq M\frac{\Gamma \left( \zeta \right) }{\Gamma \left( \zeta +\eta \right) }\left( \Psi \left( t\right) -\Psi \left( a\right) \right)^{\eta },
\end{equation}
where, $M$ is the bound of a bounded function $(\Psi(\cdot)-\Psi(a))^{1-\zeta}h(\cdot)$.
\end{lemma}
\begin{theorem}\label{lema4}
 Let $\tilde{u},$ $\tilde{v}\in L_1(\Delta,\,\mathbb{R})$ and $\tilde{g}\in C(\Delta,\,\mathbb{R}) .$ Let $\Psi \in C^{1}(\Delta,\,\mathbb{R}) $ be an increasing function with $\Psi ^{\prime }\left( t\right) \neq 0$, for all $~t\in \Delta$. Assume that
\begin{enumerate}
\item $\tilde{u}$ and $\tilde{v}$ are nonnegative;
\item $\tilde{g}$ is nonnegative and nondecreasing.

\end{enumerate}

If
\begin{equation*}
\tilde{u}( t) \leq \tilde{v}( t) +\tilde{g}( t) \int_{a}^{t}\mathrm{Q}^{\eta}_{\Psi}(t,\sigma)\,\tilde{u}(\sigma)\, \,d\sigma,
\end{equation*}
then
\begin{equation}\label{jose}
\tilde{u}(t) \leq\tilde{v}( t) +\int_{a}^{t}\overset{\infty }{%
\underset{m=1}{\sum }}\frac{\left[\tilde{g}\left( t\right) \Gamma \left( \eta
\right) \right] ^{m}}{\Gamma \left( m\eta \right) }\,\mathrm{Q}^{m\eta}_{\Psi}(t,s)\,\,\tilde{v}(\sigma) \,d\sigma,~t\in \Delta.
\end{equation}
where,\,$\mathrm{Q}^{m\eta}_{\Psi}(t,s)=\Psi'(s)(\Psi(t)-\Psi(s))^{m\eta-1}, ~t, s \in \Delta .$
Further, if  $\tilde{v}$ is  a nondecreasing function on $\Delta$ then
$$\tilde{u}(t)\leq\tilde{v}(t)\,\mathcal{E}_{\eta}\left(\tilde{g}(t) \Gamma(\eta)\left(\Psi(t)-\Psi(a)\right)^{\eta}\right),$$
where, $\mathcal{E}_{\eta}(\cdot)$ is the Mittag-Leffler function of one parameter.
\end{theorem}

Existence and uniqueness results are proved via following fixed point theorems.

\begin{theorem}[\cite{Diethelm1}, Schauder] \label{Schauder}
Let $\mathcal{X}$ be a Banach space, let $\mathcal{U}$ be a nonempty convex bounded closed   subset of $\mathcal{X}$ and let $\mathcal{A}:\mathcal{U}\rightarrow \mathcal{U}$ be a completely continuous operator. Then $\mathcal{A}$ has at least one fixed point.
\end{theorem}
\begin{theorem}[\cite{Diethelm1}, Weissinger] \label{Weissinger}
Assume $(\mathcal{U},d)$ to be a non empty complete metric space and let $\eta_j\geq 0$ for every $j\in \N_0$ such that $\overset{\infty}{\underset{j=0}{\sum }}\eta_j$ converges. Furthermore, let the mapping $\mathcal{A}:\mathcal{U}\rightarrow {\mathcal{U}}$ satisfy the inequality $$d(\mathcal{A}^{j}u,\mathcal{A}^{j}v)\leq \eta_{j}\,\,d(u,v)$$
for every $j\in \N$ and every $u,v \in \mathcal{U}.$ Then, $\mathcal{A}$ has a unique fixed point $u^{*}.$ Moreover, for any $u_{0}\in \mathcal{U},$ the sequence $(\mathcal{A}^{j}u_{0})_{j=1}^{\infty}$ converges to this fixed point $u^{*}.$
\end{theorem}

\begin{definition}[\cite{Gorenflo}]
Let $\eta >0,\, \nu >0~ ( ~\eta, \nu \in \R) $. Then the two parameter Mittag-Leffler function is defined as
$$\mathcal{E}_{\eta,\,\nu}(z)=\sum_{k=0}^{\infty}\frac{z^k}{\Gamma(k \eta +\nu)}.$$
\end{definition}
\begin{definition}[\cite{Gorenflo}]\label{defKS}
Let $\eta,\,m \in \R$ and  $l\in\C$ such that $\eta>0,~m>0$ and $\eta(jm+l)+1\neq-1,-2,-3,\cdots (j=0,1,2,\cdots) $. Then generalized (Kilbas--Saigo) Mittag--Leffler type function of three parameters is defined by  
 $$ \mathcal{E}_{\eta,\,m,\,l}\,(z)= \sum_{k=0}^{\infty} c_k\, z^k $$
where $$ c_{0} =1, ~c_{k}=\prod_{j=0}^{k-1}\frac{\Gamma(\eta[jm+l]+1)}{\Gamma(\eta[jm+l+1]+1)}\,\,(k=1,2,\cdots)$$
and an empty product is assumed to be equal to one.
\end{definition}


\section{Existence, Uniqueness and Continuous Dependence}
The forthcoming theorem  establishes the existence of solution along with the  interval of existence of the initial value problem (IVP) \eqref{eq1}-\eqref{eq2} using its equivalent  fractional Volterra integral equation(VIE)  in the weighted space $C_{1-\zeta ; \, \Psi }(\Lambda,\,\R)$.
\begin{theorem} [Existence and interval of existence]\label{th3.1}
Let $\zeta=\eta+\nu(1-\eta), 0<\eta<1$ and $0\leq\nu\leq 1$. Define
$$R_{0}=\left\{(t,y):a\leq t\leq a+\xi,\left|y-\mathcal{H}^{\Psi}_{\zeta}(t,a)\,y_a\right|\leq k \right\},~\xi>0,  ~k>0.$$
where, $\mathcal{H}^{\Psi}_{\zeta}(t,a):=\dfrac{( \Psi (t)-\Psi (a) ) ^{\zeta -1}}{\Gamma (\zeta)}$.~
Let $f:R_0\rightarrow\R$ is continuous function such that $f(\cdot\,,y(\cdot))\in C_{1-\zeta\,; \,\Psi}(\Lambda,\,\R)$ for every $y\in C_{1-\zeta\,; \,\Psi}(\Lambda,\,\R)$.
Let $\Psi\in C^1(\Lambda,\,\R)$ be an increasing bijective function with $\Psi'(t)\neq 0, $ for all $\,t \in \Lambda$ .
Then the IVP 
\eqref{eq1}-\eqref{eq2} 
possesses at least one solution $y$ in the space $C_{1-\zeta\,; \,\Psi} [a,a+\chi]$ of weighted functions,  where
$$
\chi=\min\left\{\xi,~\Psi^{-1}\left[\Psi(a)+\left(\frac{k\, \Gamma(\eta+\zeta)}{\Gamma(\zeta)\left\Vert f\right\Vert _{C_{1-\zeta \,; \,\Psi }\left[ a,a+\xi\right]}}\right)^\frac{1}{\eta}\right]-a\right\}. $$ 
\end{theorem}
\begin{proof}
The equivalent fractional VIE to the Cauchy problem \eqref{eq1}-\eqref{eq2} in the  space $C_{1-\zeta\,; \,\Psi}(\Lambda,\,\R)$ is derived in \cite{Vanterler2} and it is given by
\begin{equation}\label{eq3}
y (t) =\mathcal{H}^{\Psi}_{\zeta}(t,a)\,y_{a}+\frac{1}{\Gamma ( \eta ) }\int_{a}^{t}\mathrm{Q}^{\eta\,;\,\psi}(t,\sigma)f(\sigma,y(\sigma)) \,d\sigma, ~t \in \Lambda.
\end{equation}
Consider the set defined by,
$$
\mathcal{U}=\left\{y\in C_{1-\zeta \,; \,\Psi }\left[ a,a+\chi\right]:\left\Vert y-\mathcal{H}^{\Psi}_{\zeta}(\cdot,a)\,y_a \right\Vert_{C_{1-\zeta \,; \,\Psi }\left[ a,a+\chi\right]}\leq k\right\}.
$$

Define $\tilde{y}(t)=\mathcal{H}^{\Psi}_{\zeta}(t,a)\,y_{a}, ~t \in [ a,a+\chi]$. Then $(\Psi(t)-\Psi(a))^{1-\zeta}\, \tilde{y}(t)=\dfrac{y_a}{\Gamma(\zeta)}\in C[a,a+\chi]$ and hence  we have, $\tilde{y}(t)\in C_{1-\zeta \,; \,\Psi }\left[ a,a+\chi\right] $. Further, 
$$
\left\Vert \tilde{y}-\mathcal{H}^{\Psi}_{\zeta}(\cdot,a)\,y_a\right\Vert_{C_{1-\zeta \,; \,\Psi }\left[ a,a+\chi\right]}=0\leq k.
$$

Thus, $\tilde{y}\in \mathcal{U}$ and hence $\mathcal{U}$ is non-empty set. Clearly, $\mathcal{U}$ is a convex, bounded and closed subset 
of space $C_{1-\zeta \,; \,\Psi }\left[ a,a+\chi\right].$ 
We define an operator $\mathcal{A}$ on the set $\mathcal{U}$ by $$\mathcal{A}y(t)=\mathcal{H}^{\Psi}_{\zeta}(t,a)\,y_{a}+\frac{1}{\Gamma ( \eta ) }\int_{a}^{t}\mathrm{Q}^{\eta\,;\,\psi}(t,\sigma)f(\sigma,y(\sigma))\, \,d\sigma.$$

By definition of an operator $\mathcal{A}$, the equation \eqref{eq3} can be written as $$ y=\mathcal{A}y.$$ We analyze the properties of $\mathcal{A}$  so that it admit at least one fixed point.

First we show that, $\mathcal{A}\,\mathcal{U} \subseteq \mathcal{U}$. In the view of Lemma \ref{lema1},  $\Psi$-Riemann-Liouville fractional integral operator $\mathcal{I}_{a+}^{\eta \, ;\,\Psi }\left( \cdot \right) $ maps $C_{1-\zeta \, ;\,\Psi }\left[ a,a+\chi\right] $ to $C_{1-\zeta \, ;\,\Psi }\left[
a,a+\chi\right]$, and hence $\mathcal{I}_{a +}^{\eta\,;\, \Psi} f(\cdot,y(\cdot)) \in C_{{1-\zeta};\, \Psi}[a,a+\chi]$ for any 
$y\in C_{1-\zeta \,; \,\Psi }\left[ a,a+\chi\right]$.
Further, $$(\Psi (t)-\Psi (a) )^{1-\zeta }\,\mathcal{H}^{\Psi}_{\zeta}(t,a)\,y_{a}=\frac{y_{a}}{\Gamma ( \zeta )}\in  C[a,a+\chi],$$
and hence $\mathcal{H}^{\Psi}_{\zeta}(t,a)\,y_{a} \in C_{1-\zeta \,; \,\Psi }\left[ a,a+\chi\right]$. 

From the above arguments it follows that $ \mathcal{A}y\in C_{1-\zeta \,; \,\Psi }\left[ a,a+\chi\right]$. Next, for any $y\in \mathcal{U}$ and any $x$,~  $ a< x\leq \chi$, we have
\begin{align*}
&\left\Vert \mathcal{A}y -\mathcal{H}^{\Psi}_{\zeta}(\cdot,a)\,y_a\right\Vert _{_{C_{1-\zeta\,; \,\Psi }\left[ a,a+x\right] }}\notag\\
&=\underset{t\in \left[ a,a+x\right] }{\max }\left|(\Psi (t)-\Psi (a) )^{1-\zeta } \frac{1}{\Gamma ( \eta )}\int_{a}^{t}\mathrm{Q}^{\eta\,;\,\psi}(t,\sigma)(\Psi (\sigma)-\Psi (a) )^{\zeta-1 }\times \right.\notag\\
&\qquad\left.(\Psi (\sigma)-\Psi (a) )^{1-\zeta }f(\sigma,y(\sigma)) \,d\sigma\right|\notag\\
&\leq(\Psi (a+x)-\Psi (a) )^{1-\zeta }\frac{1}{\Gamma(\eta)}\int_{a}^{a+x}\mathrm{Q}^{\eta\,;\,\psi}(a+x,\sigma)(\Psi (\sigma)-\Psi (a) )^{\zeta-1 }\times \notag\\
&\qquad\underset{\varsigma\in \left[ a,\, \sigma\right] }{\max }\left|(\Psi (\varsigma)-\Psi (a) )^{1-\zeta }f(\varsigma,y(\varsigma))\right|\,d\sigma\notag\\
&\leq \left\Vert f \right\Vert_{C_{1-\zeta \,; \,\Psi }\left[ a,a+x\right]}(\Psi (a+x)-\Psi (a) )^{1-\zeta }\,\,\mathcal{I}_{a+}^{\eta\,; \,\Psi}(\Psi(a+x)-\Psi(a))^{\zeta-1}\notag\\
&\leq \left\Vert f \right\Vert_{C_{1-\zeta \,; \,\Psi }\left[ a,a+x\right]}(\Psi (a+x)-\Psi (a) )^{1-\zeta }\frac{\Gamma(\zeta)}{\Gamma(\eta+\zeta)}(\Psi (a+x)-\Psi (a) )^{\eta+\zeta-1 }\notag\\
&=\left\Vert f \right\Vert_{C_{1-\zeta \,; \,\Psi }\left[a,a+x\right]}\frac{\Gamma(\zeta)}{\Gamma(\eta+\zeta)}(\Psi (a+x)-\Psi (a) )^{\eta }.
\end{align*}

Taking $x \rightarrow \chi$ we get,
\begin{equation} \label{e4}
\left\Vert \mathcal{A}y -\mathcal{H}^{\Psi}_{\zeta}(\cdot,a)\,y_a \right\Vert _{_{C_{1-\zeta\,; \,\Psi }\left[ a,a+\chi\right] }} \leq \left\Vert f \right\Vert_{C_{1-\zeta \,; \,\Psi }\left[a,a+\chi\right]}\frac{\Gamma(\zeta)}{\Gamma(\eta+\zeta)}(\Psi (a+\chi)-\Psi (a) )^{\eta }
\end{equation}

Since,
$$
\chi=\min\left\{\xi\,\,,\Psi^{-1}\left[\Psi(a)+\left(\frac{k\, \Gamma(\eta+\zeta)}{\Gamma(\zeta)\left\Vert f\right\Vert _{C_{1-\zeta \,; \,\Psi }\left[ a,a+\xi\right]}}\right)^\frac{1}{\eta}\right]-a\right\} 
$$
we have
$$
\chi\leq\Psi^{-1}\left[\Psi(a)+\left(\frac{k\,\Gamma(\eta+\zeta)}{\Gamma(\zeta)\left\Vert f\right\Vert_{C_{1-\zeta\,; \,\Psi}\left[a,a+\xi\right]}}\right)^\frac{1}{\eta}\right]-a.
$$

Further, since $\Psi\in C^1\left( \Lambda,\,\R\right)  $ is a bijective function, $\Psi^{-1}: \R \to \Lambda$ exists and from above inequality, we have
\begin{equation}\label{e5}
(\Psi(a+\chi)-\Psi(a))^\eta\leq\frac{k\,\Gamma(\eta+\zeta)}{\Gamma(\zeta)\left\Vert f\right\Vert _{C_{1-\zeta\, ; \,\Psi }\left[ a,a+\xi\right]}}.
\end{equation}

Using Eq.\eqref{e5} in Eq.\eqref{e4}, we get
$$\left\Vert \mathcal{A}y -\mathcal{H}^{\Psi}_{\zeta}(\cdot,a)\,y_a\right\Vert _{_{C_{1-\zeta\,; \,\Psi }\left[ a,a+\chi\right] }}\leq k.$$
We have proved that $\mathcal{A}y\in \mathcal{U} \,\,\text{for any}\,\, y\in \mathcal{U}$. This proves $\mathcal{A} \,\,\text{maps}\,\, \mathcal{U} \,\,\text{to itself}$.  

Next, we prove that $\mathcal{A}$ is a continuous operator.  Let any $\varepsilon>0$. Since $f:R_0\rightarrow\R$ is continuous and $R_{0}$  is compact set, $f$  is uniformly continuous on $R_{0}$. Thus, there exists $\tilde{\delta}>0$ such that 
\begin{equation*}
\left|f(t,y)-f(t,z)\right|<\frac{\varepsilon\,\, \Gamma(\eta+1)}{(\Psi(a+\chi)-\Psi(a))^{\eta-\zeta+1}},~~ \text{whenever}~ |y-z|<\tilde{\delta}.
\end{equation*}

Since, $\Psi$ is continuous on compact set $\Lambda$, we can choose $\delta>0$ such that $( \Psi (t)-\Psi (a) ) ^{1-\zeta }|y-z|<\delta$. Therefore, we have
\begin{align} \label{eq34}
\left|f(t,y)-f(t,z)\right|<\frac{\varepsilon\,\, \Gamma(\eta+1)}{(\Psi(a+\chi)-\Psi(a))^{\eta-\zeta+1}},~ \text{whenever}~( \Psi (t)-\Psi (a) ) ^{1-\zeta }|y-z|<\delta.
\end{align}

Let any $y,z\in \mathcal{U} $ such that  $\left\Vert y-z \right\Vert_{C_{1-\zeta\, ; \,\Psi }\left[ a,a+\chi\right]}<\delta$. Then in the view of \eqref{eq34},  we have
$$\left|f(t,y(t))-f(t,z(t))\right|<\frac{\varepsilon\,\, \Gamma(\eta+1)}{(\Psi(a+\chi)-\Psi(a))^{\eta-\zeta+1}}, ~t\in [ a,a+\chi].$$

Therefore,
\begin{align*}
&\left\Vert \mathcal{A}y-\mathcal{A}z \right\Vert_{C_{1-\zeta\, ; \,\Psi }\left[ a,a+\chi\right]}\\
&\leq\underset{t\in \left[ a,a+\chi\right] }{\max }(\Psi (t)-\Psi (a) )^{1-\zeta }\frac{1}{\Gamma ( \eta )} \int_{a}^{t}\mathrm{Q}^{\eta\,;\,\psi}(t,\sigma)\left|f(\sigma,y(\sigma))-f(\sigma,z(\sigma))\right| \,d\sigma\\
&<\frac{\varepsilon\,\, \Gamma(\eta+1)}{(\Psi(a+\chi)-\Psi(a))^{\eta-\zeta+1}}\,\,\,\underset{t\in \left[ a,a+\chi\right] }{\max }(\Psi (t)-\Psi (a) )^{1-\zeta }\frac{1}{\Gamma ( \eta )} \int_{a}^{t}\mathrm{Q}^{\eta\,;\,\psi}(t,\sigma)\, \,d\sigma\\
&=\frac{\varepsilon\,\, \Gamma(\eta+1)}{(\Psi(a+\chi)-\Psi(a))^{\eta-\zeta+1}}\,\,\,\underset{t\in \left[ a,a+\chi\right] }{\max }(\Psi (t)-\Psi (a) )^{1-\zeta }\frac{1}{\Gamma ( \eta+1 )} ( \Psi (t) -\Psi (a)) ^{\eta} \\
&\leq\frac{\varepsilon}{(\Psi(a+\chi)-\Psi(a))^{\eta-\zeta+1}}\,\,(\Psi(a+\chi)-\Psi(a))^{\eta-\zeta+1}\\
&=\varepsilon.
\end{align*}

Thus, the operator $\mathcal{A}$ is continuous on $\mathcal{U}$. Now, for any $z\in \mathcal{A}(\mathcal{U})=\{\mathcal{A}y: y\in \mathcal{U}\}$ and any $x$ such that $a<x\leq \chi$, we have
\begin{align*}
&\left\Vert z \right\Vert _{_{C_{1-\zeta\,; \,\Psi }\left[ a,a+x\right] }}
=\left\Vert \mathcal{A}y  \right\Vert _{_{C_{1-\zeta\,; \,\Psi }\left[ a,a+x\right] }}\\
&\leq\frac{y_a}{\Gamma(\zeta)}+\underset{t\in \left[ a,a+x\right] }{\max }\left|\frac{(\Psi (t)-\Psi (a) )^{1-\zeta }}{\Gamma ( \eta ) }\int_{a}^{t}\mathrm{Q}^{\eta\,;\,\psi}(t,\sigma)\times \right.\\
&\qquad(\Psi (\sigma)-\Psi (a) )^{\zeta-1 }(\Psi (\sigma)-\Psi (a) )^{1-\zeta }f(\sigma,y(\sigma)) \,d\sigma\Big|\\
&\leq\frac{y_a}{\Gamma(\zeta)}+\frac{(\Psi (a+x)-\Psi (a) )^{1-\zeta }}{\Gamma ( \eta ) }\int_{a}^{a+x}\mathrm{Q}^{\eta\,;\,\psi}(a+x,\sigma)\times \\
&\qquad(\Psi (\sigma)-\Psi (a) )^{\zeta-1 }\underset{\varsigma \in [a,\sigma]}{\max}\left|(\Psi (\varsigma)-\Psi (a) )^{1-\zeta }f(\varsigma,y(\varsigma)) \right|\,d\sigma\\
&\leq\frac{y_a}{\Gamma(\zeta)}+{\left\Vert f \right\Vert _{_{C_{1-\zeta\,; \,\Psi }\left[ a,a+x\right] }}(\Psi (a+x)-\Psi (a) )^{1-\zeta }}\,\,\mathcal{I}_{a+}^{\eta\,; \,\Psi}(\Psi(a+x)-\Psi(a))^{\zeta-1}\\
&=\frac{y_a}{\Gamma(\zeta)}+\left\Vert f \right\Vert _{C_{1-\zeta\,; \,\Psi }\left[ a,a+x\right] }\,\,\left(\Psi (a+x)-\Psi (a)\right )^{1-\zeta }\frac{\Gamma(\zeta)}{\Gamma(\eta+\zeta)}(\Psi (a+x)-\Psi (a) )^{\eta+\zeta-1 }\\
&=\frac{y_a}{\Gamma(\zeta)}+\left\Vert f \right\Vert _{C_{1-\zeta\,; \,\Psi }\left[ a,a+x\right] }\,\,\frac{\Gamma(\zeta)}{\Gamma(\eta+\zeta)}(\Psi (a+x)-\Psi (a) )^{\eta }.
\end{align*}

This shows that $\mathcal{A}(\mathcal{U})$  is uniformly bounded subset of $C_{1-\zeta\,; \,\Psi }\left[ a,a+\chi\right] $  and hence it is point wise bounded also.  Next, we prove that $\mathcal{A}(\mathcal{U})$  is equicontinuous. For any $t_1,t_2$ such that $a<t_1\leq t_2 \leq a+\chi$ and any $y\in \mathcal{U},$ we have
\begin{align*}
&|\mathcal{A}y(t_2)-\mathcal{A}y(t_1)|\\
&\leq\left|\frac{y_a}{\Gamma(\zeta)}\left\{\left( \Psi (t_2)-\Psi (a) \right) ^{\zeta -1}-( \Psi (t_1)-\Psi (a) ) ^{\zeta -1}\right\} \right|\\
& \,\,+\left|\frac{1}{\Gamma(\eta)}\int_{a}^{t_2}\mathrm{Q}^{\eta\,;\,\psi}(t_2,\sigma)\left|f(\sigma,y(\sigma))\right|\,d\sigma-\frac{1}{\Gamma(\eta)}\int_{a}^{t_1}\mathrm{Q}^{\eta\,;\,\psi}(t_1,\sigma)\left|f(\sigma,y(\sigma))\right|\,d\sigma\right|\\
&\leq\left|\frac{y_a}{\Gamma(\zeta)}\left\{\left( \Psi (t_2)-\Psi (a) \right) ^{\zeta -1}-( \Psi (t_1)-\Psi (a) ) ^{\zeta -1}\right\} \right|\\
&\qquad+\left|\frac{1}{\Gamma(\eta)}\int_{a}^{t_2}\mathrm{Q}^{\eta\,;\,\psi}(t_2,\sigma)(\Psi (\sigma) -\Psi (a)) ^{\zeta-1}\underset{\varsigma\in [a,\sigma]}{\max}\left|( \Psi (\varsigma) -\Psi (a)) ^{1-\zeta}f(\varsigma,y(\varsigma))\right|\,d\sigma\right.\\
&\qquad\left.-\frac{1}{\Gamma(\eta)}\int_{a}^{t_1}\mathrm{Q}^{\eta\,;\,\psi}(t_1,\sigma)(\Psi (\sigma) -\Psi (a)) ^{\zeta-1} \underset{\varsigma\in [a,\sigma]}{\max}\left|( \Psi (\varsigma) -\Psi (a)) ^{1-\zeta}f(\varsigma,y(\varsigma))\right|\,d\sigma\right|\\
&\leq\left|\frac{y_a}{\Gamma(\zeta)}\left\{\left( \Psi (t_2)-\Psi (a) \right) ^{\zeta -1}-( \Psi (t_1)-\Psi (a) ) ^{\zeta -1}\right\} \right|\\
&\qquad+\left\Vert f \right\Vert_{C_{1-\zeta \,; \,\Psi }[a,a+\chi]}\left|\mathcal{I}_{a+}^{\eta\,; \,\Psi}(\Psi(t_2)-\Psi(a))^{\zeta-1}-\mathcal{I}_{a+}^{\eta\,; \,\Psi}(\Psi(t_1)-\Psi(a))^{\zeta-1}\right|\\
&=\frac{y_a}{\Gamma(\zeta)}\left|\left( \Psi (t_2)-\Psi (a) \right) ^{\zeta -1}-( \Psi (t_1)-\Psi (a) ) ^{\zeta -1} \right|\\
&\qquad+\left\Vert f \right\Vert_{C_{1-\zeta \,; \,\Psi }[a,a+\chi]}\frac{\Gamma(\zeta)}{\Gamma(\eta+\zeta)}\left|(\Psi (t_2)-\Psi (a) )^{\eta+\zeta-1 }-(\Psi (t_1)-\Psi (a) )^{\eta+\zeta-1 }\right|.
\end{align*}

Observe that, the right hand  part in the preceding  inequality is free from $y$. Thus, using the continuity of $\Psi$, $|t_2-t_1|\rightarrow 0  $  implies that  $|\mathcal{A}y(t_2)-\mathcal{A}y(t_1)|\rightarrow 0$. This proves that $\mathcal{A}(\mathcal{U})$  is equicontinuous. In the view of  Arzela-Ascoli Theorem \cite{Diethelm1}, it follows that $\mathcal{A}(\mathcal{U})$ is relatively compact. Therefore, by Schauder's fixed point Theorem \ref{Schauder},  operator $\mathcal{A}$ has at least one fixed point $y \in C_{1-\zeta\,; \,\Psi }\left[ a,a+\chi\right] $ which is the solution of IVP \eqref{eq1}--\eqref{eq2}.
\end{proof}

In the next theorem, using  Lipschitz condition on $f$ and  the Weissinger theorem,  we establish another existence result for the IVP \eqref{eq1}--\eqref{eq2}, which gives in addition the uniqueness of solution also. 
\begin{theorem}[Uniqueness] \label{th4.1}
Let $\zeta=\eta+\nu(1-\eta),0<\eta<1$, $0\leq\nu\leq 1$  and 
$$R_{0}=\left\{(t,y):a\leq t\leq a+\xi,\left|y-\mathcal{H}^{\Psi}_{\zeta}(t,a)\,y_{a}\right|\leq k \right\},~\xi>0,  ~k>0.$$
where, $\mathcal{H}^{\Psi}_{\zeta}(t,a)=\frac{\left( \Psi (t)-\Psi (a) \right) ^{\zeta-1}}{\Gamma(\zeta)}$.
Let $f:R_0\rightarrow\R$ be a function such that $f(\cdot\,,y(\cdot))\in C_{1-\zeta\,; \,\Psi} (\Lambda,\,\R)$ for every $y\in C_{1-\zeta\,; \,\Psi}(\Lambda,\,\R)$ and  satisfies the Lipschitz condition 
$$
\left | f(t,u)-f(t,v)\right|\leq L|u-v|,\,\,u,v\in \R, ~L>0.
$$

Let $\Psi\in C^1(\Lambda,\,\R)$ be an increasing bijective function with $\Psi'(t)\neq 0 $ for all $t\in \Lambda$. Then, the IVP 	
\eqref{eq1}-\eqref{eq2} has a unique solution
$y$ in the weighted space $C_{1-\zeta\,; \,\Psi} [a,a+\chi]$,  where
$$
\chi=\min\left\{\xi,~\Psi^{-1}\left[\Psi(a)+\left(\frac{k\, \Gamma(\eta+\zeta)}{\Gamma(\zeta)\left\Vert f\right\Vert _{C_{1-\zeta \,; \,\Psi }\left[ a,a+\xi\right]}}\right)^\frac{1}{\eta}\right]-a\right\}. 
$$
\end{theorem}
\begin{proof}  Consider the operator $\mathcal{A}$ used in Theorem \ref{th3.1}. It is as of now demonstrated that $\mathcal{A}:\mathcal{U}\rightarrow \mathcal{U}$  is continuous on closed, convex and bounded subset $\mathcal{U}$ of Banach space $C_{1-\zeta\,; \,\Psi }\left[ a,a+\chi\right]$. We now prove that, for every $n\in \N$ and every $x$ such that $a<x\leq \chi$, we have
\begin{equation}\label{e6}
\left\Vert \mathcal{A}^{n}y -\mathcal{A}^{n}z \right\Vert _{_{C_{1-\zeta\,; \,\Psi }\left[ a,a+x\right] }}\leq\frac{\Gamma(\zeta)}{\Gamma(n \eta+\zeta)}\left[L\left(\Psi (a+x)-\Psi (a) \right)^{\eta } \right]^n\left\Vert y -z \right\Vert _{_{C_{1-\zeta\,; \,\Psi }\left[ a,a+x\right]}}.
\end{equation} 

We provide the proof of above inequality by using mathematical induction. By using Lipschitz condition on $f$, for any $x \in (a,\chi]$, we have 
\begin{align*}
&\left\Vert \mathcal{A}y -\mathcal{A}z\right\Vert_{{C_{1-\zeta\,; \,\Psi}\left[a,a+x\right]}}\\
&\leq(\Psi (a+x)-\Psi (a) )^{1-\zeta }\frac{1}{\Gamma ( \eta ) }\int_{a}^{a+x}\mathrm{Q}^{\eta\,;\,\psi}(a+x,\sigma)\left|f(\sigma,y(\sigma))-f(\sigma,z(\sigma))\right| \,d\sigma\\
&\leq L(\Psi (a+x)-\Psi (a) )^{1-\zeta }\frac{1}{\Gamma ( \eta ) }\int_{a}^{a+x}\mathrm{Q}^{\eta\,;\,\psi}(a+x,\sigma)\times  \\
&\qquad(\Psi (\sigma) -\Psi (a)) ^{\zeta-1}(\Psi (\sigma) -\Psi (a)) ^{1-\zeta}\left|y(\sigma)-z(\sigma)\right| \,d\sigma\\
&\leq L(\Psi (a+x)-\Psi (a) )^{1-\zeta }\frac{1}{\Gamma ( \eta ) }\int_{a}^{a+x}\mathrm{Q}^{\eta\,;\,\psi}(a+x,\sigma)\times  \\
&\qquad(\Psi (\sigma) -\Psi (a)) ^{\zeta-1}\underset{\varsigma\in \left[ a,\sigma\right] }{\max }\left|(\Psi (\varsigma) -\Psi (a)) ^{1-\zeta}[y(\varsigma)-z(\varsigma)]\right| \,d\sigma\\
&\leq L\left\Vert y-z \right\Vert_{C_{1-\zeta \,; \,\Psi }\left[ a,a+x\right]}(\Psi (a+x)-\Psi (a) )^{1-\zeta }\,\,\mathcal{I}_{a+}^{\eta\,; \,\Psi}(\Psi(a+x)-\Psi(a))^{\zeta-1}\\
&=\frac{\Gamma(\zeta)}{\Gamma(\eta+\zeta)}\left[L(\Psi (a+x)-\Psi (a) )^{\eta }\right]\left\Vert y-z \right\Vert_{C_{1-\zeta \,; \,\Psi }\left[ a,a+x\right]}.
\end{align*}

Hence, the inequality \eqref{e6} is true for $n=1$. Let us assume that, it is true for $n=k-1$. Now, we prove that inequality \eqref{e6} is also true for $n=k$. Then, we have
\begin{align*}
&\left\Vert \mathcal{A}^{k}y -\mathcal{A}^{k}z \right\Vert _{_{C_{1-\zeta\,; \,\Psi }\left[ a,a+x\right] }}\\
&=\underset{t\in \left[ a,a+x\right] }{\max }\left|(\Psi (t)-\Psi (a) )^{1-\zeta }\frac{1}{\Gamma ( \eta ) }\int_{a}^{t}\mathrm{Q}^{\eta\,;\,\psi}(t,\sigma)
\left[f\left(\sigma,\mathcal{A}^{k-1}y(\sigma)\right)-f\left(\sigma,\mathcal{A}^{k-1}z(\sigma)\right)\right] \,d\sigma\right|\\
&\leq L\,\,(\Psi (a+x)-\Psi (a) )^{1-\zeta }\frac{1}{\Gamma ( \eta ) }\int_{a}^{a+x}\mathrm{Q}^{\eta\,;\,\psi}(a+x,\sigma) \left|\mathcal{A}^{k-1}y(\sigma)-\mathcal{A}^{k-1}z(\sigma)\right|\,d\sigma\\
&\leq L\,\,(\Psi (a+x)-\Psi (a) )^{1-\zeta }\frac{1}{\Gamma ( \eta ) }\int_{a}^{a+x}\mathrm{Q}^{\eta\,;\,\psi}(a+x,\sigma)\times \\ 
& \qquad(\Psi (\sigma) -\Psi (a)) ^{\zeta-1}\left\Vert \mathcal{A}^{k-1}y-\mathcal{A}^{k-1}z \right\Vert _{_{C_{1-\zeta\,; \,\Psi }\left[ a,\sigma\right] }} \,d\sigma\\
&\leq L\,\,(\Psi (a+x)-\Psi (a) )^{1-\zeta }\frac{1}{\Gamma ( \eta ) }\int_{a}^{a+x}\mathrm{Q}^{\eta\,;\,\psi}(a+x,\sigma)(\Psi (\sigma) -\Psi (a)) ^{\zeta-1}\times \\
&\qquad \frac{\Gamma(\zeta)}{\Gamma((k-1) \eta+\zeta)}\left[L\left(\Psi (\sigma)-\Psi (a)\right)^{\eta } \right]^{k-1}\left\Vert y -z \right\Vert _{_{C_{1-\zeta\,; \,\Psi }\left[ a,\sigma\right] }}\,d\sigma\\
&\leq L^k\frac{\Gamma(\zeta)}{\Gamma((k-1) \eta+\zeta)}\,\,(\Psi (a+x)-\Psi (a) )^{1-\zeta }\left\Vert y -z \right\Vert _{_{C_{1-\zeta\,; \,\Psi }\left[ a,a+x\right] }} \times\\
&\qquad \mathcal{I}_{a+}^{\eta\,; \,\Psi}(\Psi(a+x)-\Psi(a))^{\eta(k-1)+\zeta-1}\\
&= L^k\frac{\Gamma(\zeta)}{\Gamma((k-1) \eta+\zeta)}\,\,(\Psi (a+x)-\Psi (a) )^{1-\zeta }\left\Vert y -z \right\Vert _{_{C_{1-\zeta\,; \,\Psi }\left[ a,a+x\right] }} \times\\
&\qquad\frac{\Gamma((k-1) \eta+\zeta)}{\Gamma(k\eta+\zeta)}(\Psi (a+x)-\Psi (a) )^{k \eta+\zeta-1}\\
&= \frac{\Gamma(\zeta)}{\Gamma(k \eta+\zeta)}\left[L(\Psi (a+x)-\Psi (a) )^\eta\right]^k\left\Vert y -z \right\Vert _{_{C_{1-\zeta\,; \,\Psi }\left[ a,a+x\right] }}.
\end{align*}

Hence, the inequality \eqref{e6} is true  for $n=k$. By  mathematical induction the proof of the inequality \eqref{e6} is concluded. Taking $x \to \chi$, we obtain
\begin{equation*}
\left\Vert \mathcal{A}^{n}y -\mathcal{A}^{n}z \right\Vert _{_{C_{1-\zeta; \,\psi }\left[ a,a+\chi\right] }}\leq\frac{\Gamma(\zeta)}{\Gamma(n \eta+\zeta)}\left[L\left(\psi (a+\chi)-\psi (a)\right)^{\eta } \right]^n\left\Vert y -z \right\Vert _{_{C_{1-\zeta; \,\psi }\left[ a,a+\chi\right]}}.
\end{equation*}

Note that, the operator $\mathcal{A}$ satisfy all the conditions  of the  Weissinger's theorem \ref{Weissinger}, with 
$$
\eta_n=\frac{\left[L\left(\Psi (a+\chi)-\Psi (a)\right)^{\eta } \right]^n}{\Gamma(n \eta+\zeta)}.
$$ 

Further, by definition of two parameter Mittag--Leffler function, we have
 $$\sum_{n=0}^{\infty}  \,\eta_n= \sum_{n=0}^{\infty}  \frac{\left[L(\Psi (a+\chi)-\Psi (a) )^{\eta } \right]^n}{\Gamma(n \eta+\zeta)}=\mathcal{E}_{\eta,\,\zeta}\,\left(L(\Psi (a+\chi)-\Psi (a) )^{\eta }\right),$$
 which is convergent series. Thus, by using Weissinger's fixed point theorem \ref{Weissinger}, operator $\mathcal{A}$ has a unique fixed point $y$ in   $C_{1-\zeta\,; \,\Psi }\left[ a,a+\chi\right] $, which is a unique solution of the IVP \eqref{eq1}-\eqref{eq2}.
\end{proof}

\begin{rem}
\begin{enumerate}
\item The results obtained by Diethelem and Ford {\rm \cite{Diethelm1}} can be regarded as the  particular cases of  {\rm Theorem \ref{th3.1}} and {\rm Theorem \ref{th4.1} }and can be derived by putting $\nu=1,~a=0 $ and  $\Psi(t)=t, ~t\in\Lambda$.

\item We have investigated in {\rm Theorem \ref{th3.1}} and {\rm Theorem \ref{th4.1}}, the existence and uniqueness of the Cauchy problem {\rm (\ref{eq1})}-{\rm(\ref{eq2})} involving the $\Psi$-Hilfer fractional derivative. One of the fundamental properties of the $\Psi$-Hilfer fractional derivative is the wide class of fractional derivatives that contain it as particular cases. Thus, in this sense, the respective results obtained by {\rm Theorem \ref{th3.1}} and {Theorem \ref{th4.1}}, are also valid for these fractional derivatives. For examples, some particular cases in the items as follows:

\begin{enumerate}
\item If we take  $\nu\rightarrow 1$ on both sides of  {\rm(\ref{eq1})}-{\rm(\ref{eq2})},it reduces to the Cauchy problem with $\Psi$-Caputo fractional derivative and consequently, the {\rm Theorem \ref{th3.1}} and {\rm Theorem \ref{th4.1}}, there are true;

\item  If we take $\nu\rightarrow 0$ on both sides of the {\rm(\ref{eq1})}-{\rm(\ref{eq2})}, we have the Cauchy problem with $\Psi$-Riemann-Liouville fractional derivative and consequently, the {\rm Theorem \ref{th3.1}} and {\rm Theorem \ref{th4.1}}, there are true;

\item Choose $\Psi(t)=t^{\rho}$ and take $\nu\rightarrow 1$ on both sides of  {\rm(\ref{eq1})}-{\rm(\ref{eq2})}, we have the Cauchy problem involving the Caputo-Katugampola fractional derivative and consequently, the {\rm Theorem \ref{th3.1}} and {\rm Theorem \ref{th4.1}}, there are true;

\item In cases where they involve the fractional derivatives of Hadamard, Caputo-Hadamard, Hilfer-Hadamard or any other that is related to Hadamard's fractional derivative, we have to impose the condition on the parameter $a>0$, since in these fractional derivatives they involve the function $\ln t$ and is not defined when $t=a=0$.
\end{enumerate}
\end{enumerate}
\end{rem}


Next, we prove the continuous dependence of solution for the Cauchy type problem  \eqref{eq1}-\eqref{eq2} via Weissinger's theorem.
\begin{theorem} [Continuous Dependence]\label{th5.1}
Let $\zeta =\eta +\nu \left( 1-\eta \right) $ where, $0<\eta <1$ and  $0\leq \nu \leq 1.$ Let $f:\left[ a,a +\chi \right] \times \mathbb{R}\rightarrow \mathbb{R}
$ be a function such that $f(\cdot\,,y(\cdot))\in C_{1-\zeta\,; \,\Psi} [a,a+\chi]$ for every $y\in C_{1-\zeta\,; \,\Psi} [a,a+\chi]$ and  satisfies the Lipschitz condition 
$$
\left | f(t,u)-f(t,v)\right|\leq L|u-v|,\,\,u,v\in \R, ~L>0.
$$

Let $y(t)$ and $z(t)$ be the solutions of, the IVPs,
\begin{align}
{}^H\mathcal{D}_{a+}^{\eta,\,\nu; \,\Psi}y(t)=f(t,y(t) ), \quad \mathcal{I}_{a+}^{1-\zeta\,; \,\Psi}y(a)=y_a \label{e13}
\end{align} 
 and
\begin{align}
{}^H\mathcal{D}_{a+}^{\eta,\,\nu; \,\Psi}z(t)=f(t,z(t) ), \quad \mathcal{I}_{a+}^{1-\zeta\,; \,\Psi}z(a)=z_a \label{e14}
\end{align} 
respectively. Then, 
\begin{align}\label{cd}
\left\Vert y-z\right\Vert_{C_{1-\zeta\,;\,\Psi}[a,a+\chi]}\leq \left\{1+\Gamma(\zeta)\,\mathcal{E}_{\eta,\,\zeta}\left(L\left(\Psi(a+\chi)-\Psi(a)\right)^\eta \right)\right\} \left\Vert \tilde{y}_{a}-\tilde{z}_{a}\right\Vert_{C_{1-\zeta\,;\,\Psi}[a,a+\chi]},
\end{align}
where,
$\tilde{y}_{a}(t)=\dfrac{\left( \Psi (t)-\Psi (a) \right) ^{\zeta -1}}{\Gamma ( \zeta ) }y_{a}$ and $\tilde{z}_{a}(t)=\dfrac{\left( \Psi (t)-\Psi (a) \right) ^{\zeta -1}}{\Gamma ( \zeta ) }z_{a}$.
\end{theorem}
\begin{proof}
Consider the sequences $\{\mathcal{A}^{m}\tilde{y}_{a}\}$ and $\{\mathcal{A}^{m}\tilde{z}_{a}\}$ defined by 
\begin{align*}
& \mathcal{A}^{0}\tilde{y}_{a}(t)=\tilde{y}_{a}(t),\\ 
& \mathcal{A}^{m}\tilde{y}_{a}(t)=\tilde{y}_{a}(t)+\frac{1}{\Gamma ( \eta ) }\int_{a}^{t}\mathrm{Q}^{\eta\,;\,\psi}(t,\sigma)f\left(\sigma,\mathcal{A}^{m-1}\tilde{y}_{a}(\sigma)\right) \,d\sigma, ~(m=1,2,\cdots)
\end{align*}
and
\begin{align*}
& \mathcal{A}^{0}\tilde{z}_{a}(t)=\tilde{z}_{a}(t),\\ 
& \mathcal{A}^{m}\tilde{z}_{a}(t)=\tilde{z}_{a}(t)+\frac{1}{\Gamma ( \eta ) }\int_{a}^{t}\mathrm{Q}^{\eta\,;\,\psi}(t,\sigma)f\left(\sigma,\mathcal{A}^{m-1}\tilde{z}_{a}(\sigma)\right) \,d\sigma, ~(m=1,2,\cdots)\\respectively.
\end{align*}
Then, for each $x$ with  $a<x\leq \chi$,
\begin{align*}
&\left\Vert \mathcal{A}^{m}\tilde{y}_{a}-\mathcal{A}^{m}\tilde{z}_{a}\right\Vert_{C_{1-\zeta\,; \,\Psi}[a,a+x]}\\
&=\underset{t\in\left[a,a+x\right]}{\max}\left|(\Psi(t)-\Psi(a))^{1-\zeta}
\Big\{\tilde{y}_{a}(t)-\tilde{z}_{a}(t)\right.\\
&\left.\left.\qquad +\frac{1}{\Gamma(\eta)}\int_{a}^{t}\mathrm{Q}^{\eta\,;\,\psi}(t,\sigma)\left[f\left(\sigma,\mathcal{A}^{m-1}\tilde{y}_{a}(\sigma)\right)-f\left(\sigma,\mathcal{A}^{m-1}\tilde{z}_{a}(\sigma)\right)\right] \,d\sigma \right\}\right|\\
&\leq \left\Vert \tilde{y}_{a}-\tilde{z}_{a}\right\Vert_{C_{1-\zeta\,; \,\Psi}[a,a+x]}+(\Psi(a+x)-\Psi(a))^{1-\zeta}\frac{L}{\Gamma(\eta)}\int_{a}^{a+x}\mathrm{Q}^{\eta\,;\,\psi}(a+x,\sigma)\times\\
&\qquad( \Psi (\sigma) -\Psi (a))^{\zeta-1}\underset{\varsigma\in\left[a,\sigma\right]}{\max}\left|(\Psi (\varsigma) -\Psi (a))^{1-\zeta}\left\{\mathcal{A}^{m-1}\tilde{y}_{a}(\varsigma)-\mathcal{A}^{m-1}\tilde{z}_{a}(\varsigma)  \right\}  \right|\,d\sigma\\
&\leq \left\Vert \tilde{y}_{a}-\tilde{z}_{a}\right\Vert_{C_{1-\zeta\,; \,\Psi}[a,a+x]}+L(\Psi(a+x)-\Psi(a))^{1-\zeta}\frac{1}{\Gamma(\eta)}\int_{a}^{a+x}\mathrm{Q}^{\eta\,;\,\psi}(a+x,\sigma)\times\\
&\qquad( \Psi (\sigma) -\Psi (a))^{\zeta-1}\left\Vert \mathcal{A}^{m-1}\tilde{y}_{a}-\mathcal{A}^{m-1}\tilde{z}_{a}\right\Vert_{C_{1-\zeta\,; \,\Psi}[a,\sigma]}\,d\sigma. 
\end{align*}

Using the inequality \eqref{e6},  
\begin{align*}
&\left\Vert \mathcal{A}^{m}\tilde{y}_{a}-\mathcal{A}^{m}\tilde{z}_{a}\right\Vert_{C_{1-\zeta\,; \,\Psi}[a,a+x]}\\
&\leq \left\Vert \tilde{y}_{a}-\tilde{z}_{a}\right\Vert_{C_{1-\zeta\,; \,\Psi}[a,a+x]}+(\Psi(a+x)-\Psi(a))^{1-\zeta}\frac{L}{\Gamma(\eta)}\int_{a}^{a+x}\mathrm{Q}^{\eta\,;\,\psi}(a+x,\sigma)\times\\
&\qquad( \Psi (\sigma) -\Psi (a))^{\zeta-1}\frac{\Gamma(\zeta)}{\Gamma((m-1) \eta+\zeta)}\left[L\left(\Psi (\sigma)-\Psi (a) \right)^{\eta } \right]^{m-1}\left\Vert \tilde{y}_a -\tilde{z}_a \right\Vert _{_{C_{1-\zeta\,; \,\Psi }\left[ a,\sigma\right] }}\,d\sigma\\
&= \left\Vert \tilde{y}_{a}-\tilde{z}_{a}\right\Vert_{C_{1-\zeta\,; \,\Psi}[a,a+x]}
+L^{m}\frac{\Gamma(\zeta)}{\Gamma((m-1) \eta+\zeta)}\left\Vert \tilde{y}_{a}-\tilde{z}_{a}\right\Vert_{C_{1-\zeta\,; \,\Psi}[a,a+x]}\times\\
&\qquad(\Psi(a+x)-\Psi(a))^{1-\zeta}\,\,\mathcal{I}_{a+}^{\eta\,; \,\Psi}( \Psi (a+x) -\Psi (a)) ^{(m-1)\eta+\zeta-1}\\
&= \left\Vert \tilde{y}_{a}-\tilde{z}_{a}\right\Vert_{C_{1-\zeta\,; \,\Psi}[a,a+x]}\left\{1+L^{m}\frac{\Gamma(\zeta)}{\Gamma((m-1) \eta+\zeta)}(\Psi(a+x)-\Psi(a))^{1-\zeta}\times
\right.\\
&\left.\qquad\frac{\Gamma((m-1) \eta+\zeta)}{\Gamma(m \eta+\zeta)} (\Psi(a+x)-\Psi(a))^{m\eta+\zeta-1}\right\}\\
&=\left\{1+{\Gamma(\zeta)}\frac{\left[L(\Psi(a+x)-\Psi(a))^{\eta}\right]^m}{\Gamma(m \eta+\zeta)}\right\}\left\Vert \tilde{y}_{a}-\tilde{z}_{a}\right\Vert_{C_{1-\zeta\,; \,\Psi}[a,a+x]}\\
& \leq \left\{1+{\Gamma(\zeta)}\sum_{k=0}^{m}\frac{\left[L(\Psi(a+x)-\Psi(a))^{\eta}\right]^m}{\Gamma(m \eta+\zeta)}\right\}\left\Vert \tilde{y}_{a}-\tilde{z}_{a}\right\Vert_{C_{1-\zeta\,; \,\Psi}[a,a+x]}.
\end{align*}

Taking limit as $ m\to \infty $ and $ x\to \chi $ and utilizing Theorem \ref{Weissinger}, we get the desired inequality \eqref{cd}.
\end{proof}

\begin{rem}
In the above theorem, in particular if $y_a=z_a$, then we get uniqueness of solution.
\end{rem}

\section{Picard's Successive Approximations: Nonlinear Case}
In this section, we define the Picard's type successive approximations and prove that it converges to a unique solution of nonlinear Cauchy problem \eqref{eq1}-\eqref{eq2}. Further, we obtain an estimation for the error bound.

\begin{theorem} \label{th4.2}
Let $\zeta=\eta+\nu(1-\eta), 0<\eta<1$ and $0\leq\nu\leq 1$. Define
$$R_{0}=\left\{(t,y):a\leq t\leq a+\xi,\left|y-\mathcal{H}^{\Psi}_{\zeta}(t,a)\,y_{a}\right|\leq k \right\},~\xi>0,  ~k>0.$$
where, $\mathcal{H}^{\Psi}_{\zeta}(t,a)=\frac{\left( \Psi (t)-\Psi (a) \right) ^{\zeta-1}}{\Gamma(\zeta)}$.
Let $f(\cdot\,,y(\cdot))\in C_{1-\zeta\,; \,\Psi} [a,a+\chi]$ for every \,$y\in C_{1-\zeta\,; \,\Psi} [a,a+\chi]$. 
Let $\Psi\in C^1\left([a,a+\chi],~\R \right)$ be an increasing bijective function with $\Psi'(t)\neq 0$, for all $\, t \in [a,a+\chi]$,  where
$$
\chi=\min\left\{\xi,~\Psi^{-1}\left[\Psi(a)+\left(\frac{k\, \Gamma(\eta+\zeta)}{\Gamma(\zeta)M}\right)^\frac{1}{\eta}\right]-a\right\} 
$$ 
and $M$ is constant such that
$$
|\left( \Psi (t)-\Psi (a) \right) ^{1-\zeta}f(t)|\leq M,~ \forall \,t \in [a,a+\chi].	
$$
Then the successive approximations defined by 
\begin{align}\label{e7}
 y_0(t)&=\mathcal{H}^{\Psi}_{\zeta}(t,a)\,y_{a}\notag \\
 y_n(t)&=\mathcal{H}^{\Psi}_{\zeta}(t,a)\,y_{a} +\mathrm{I}_{a+}^{\eta\,; \,\Psi}f\left(t,y_{n-1}(t)\right), n=1,2,\cdots.
 \end{align}
 satisfies the following conditions:
 \begin{enumerate}
\item [{\rm(a)}]$y_{n}\in C_{1-\zeta;\, \Psi}[a,a+\chi], ~n=0,1,2,\cdots$ ;
\item [{\rm(b)}] $(\Psi(t)-\Psi(a))^{1-\zeta}|y_{n}(t)-y_0(t)|\leq \dfrac{M\, \Gamma(\zeta)}{\Gamma(\eta+\zeta)}(\Psi(t)-\Psi(a))^{\eta}, ~t \in [a, a+\chi], ~n=0,1,2,\cdots.$
 \end{enumerate}
\end{theorem}
\begin{proof} 
(a) Since $y_0(t)=\mathcal{H}^{\Psi}_{\zeta}(t,a)\,y_{a}$, we have  $( \Psi (t)-\Psi (a) ) ^{1-\zeta}y_0(t)=\dfrac{y_a}{\Gamma(\zeta)} \in C[a,a+\chi]$ and hence $y_0\in C_{1-\zeta\,; \,\Psi}[a,a+\chi]$. 
Further, by Lemma \ref{lema1},  $\Psi$--Riemann fractional integral operator $\mathcal{I}_{a+}^{\eta \,; \,\Psi }\left( \cdot \right) $ maps $C_{1-\zeta \,; \,\Psi }\left[ a,a+\chi\right] $ to $C_{1-\zeta \,; \,\Psi }\left[
a,a+\chi\right]$ and hence $\mathcal{I}_{a +}^{\eta;\, \Psi} f(\cdot,y_{n-1}(\cdot)) \in C_{{1-\zeta};\, \Psi}[a,a+\chi]$ for any 
$y_{n-1}\in C_{1-\zeta \,; \,\Psi }\left[ a,a+\chi\right]$ for $n=1,2,\cdots$.
This proves, $y_n\in C_{1-\zeta\,; \,\Psi}[a,a+\chi]~ \forall \, n\in \N$.\\

(b) For any $t \in [a, a+\chi] $ and $n \in \N $, we have
\begin{align*}
|y_{n}(t)-y_{0}(t)|&=\left|\frac{1}{\Gamma(\eta)}\int_{a}^{t} \mathrm{Q}^{\eta}_{\Psi}(t,\sigma)f(\sigma, y_{n-1}(\sigma))\,d\sigma\right|\\
&\leq \frac{1}{\Gamma(\eta)}\int_{a}^{t} \mathrm{Q}^{\eta}_{\Psi}(t,\sigma)(\Psi(\sigma)-\Psi(a))^{\zeta-1}\times\\
&\qquad\left|(\Psi(\sigma)-\Psi(a))^{1-\zeta}f(\sigma, y_{n-1}(\sigma))\right|\,d\sigma\\
&\leq M\,\,\mathcal{I}_{a+}^{\eta\,; \,\Psi}(\Psi(t)-\Psi(a))^{\zeta-1}\\
&= M \frac{\Gamma(\zeta)}{\Gamma(\eta+\zeta)}(\Psi(t)-\Psi(a))^{\eta+\zeta-1}.
\end{align*}
Therefore,
\begin{equation*}
(\Psi(t)-\Psi(a))^{1-\zeta}|y_{n}(t)-y_0(t)|\leq \frac{M\, \Gamma(\zeta)}{\Gamma(\eta+\zeta)}(\Psi(t)-\Psi(a))^{\eta}, ~t \in [a, a+\chi], ~n=0,1,2,\cdots.
\end{equation*}
\end{proof}
\begin{theorem}\label{theorem4.2}
Let $\zeta=\eta+\nu(1-\eta), 0<\eta<1$ and $0\leq\nu\leq 1$. Define
$$R_{0}=\left\{(t,y):a\leq t\leq a+\xi,\left|y-\mathcal{H}^{\Psi}_{\zeta}(t,a)\,y_{a}\right|\leq k \right\},~\xi>0,  ~k>0.$$
where, $\mathcal{H}^{\Psi}_{\zeta}(t,a)=\frac{\left( \Psi (t)-\Psi (a) \right) ^{\zeta-1}}{\Gamma(\zeta)}$.
Let $f(\cdot\,,y(\cdot))\in C_{1-\zeta\,; \,\Psi} [a,a+\chi]$ for every $y\in C_{1-\zeta\,; \,\Psi} [a,a+\chi]$ and  satisfies the Lipschitz condition 
$$
\left | f(t,u)-f(t,v)\right|\leq L|u-v|,\,\,u,v\in \R, ~L>0.
$$ 
Let $\Psi\in C^1\left([a,a+\chi],~\R \right)$ be an increasing bijective function with $\Psi'(t)\neq 0$,for all $\, t \in [a,a+\chi]$,  where
$$
\chi=\min\left\{\xi,~\Psi^{-1}\left[\Psi(a)+\left(\frac{k\, \Gamma(\eta+\zeta)}{\Gamma(\zeta)M}\right)^\frac{1}{\eta}\right]-a\right\} 
$$ 
and $M$ is constant such that
$$
|\left( \Psi (t)-\Psi (a) \right) ^{1-\zeta}f(t)|\leq M,~ \forall \,t \in [a,a+\chi].	
$$
Then the successive approximations defined by \eqref{e7}
converges to the unique solution $y$ of the Cauchy type problem	
\eqref{eq1}-\eqref{eq2} in the weighted space $C_{1-\zeta\,; \,\Psi} [a,a+\chi]$.
\end{theorem}
\begin{proof} We give the proof of the theorem in the following steps.
\noindent
\\

{{Step 1}:} In the above theorem we have already proved that $y_n\in C_{1-\zeta\,; \,\Psi}[a,a+\chi]$,~$\forall \, n\in \N$. 
\noindent \\

{{Step 2}:} We show that the sequence $\{y_n\}$ converges  to a function $y\in C_{1-\zeta\,; \,\Psi}[a,a+\chi]$ with respect to the norm $\|\cdot\|_{C_{1-\zeta\,; \,\Psi}[a,a+\chi]}$.
Observe that $y_n$ can be written as
 \begin{equation}\label{e8}
y_n=y_0+\sum_{k=1}^{n}(y_k - y_{k-1})
\end{equation} 
 which is partial sum of series 
 \begin{equation}\label{e9}
y_0+\sum_{k=1}^{\infty}(y_k - y_{k-1}).
 \end{equation} 

Therefore, to show that sequence $\{y_n\}$ is convergent, we prove that the series \eqref{e9} is convergent. For any $x \in [a,a+\chi]$, consider the space $C_{1-\zeta\,; \,\Psi}[a,x]$ having norm
\begin{equation}\label{space1}
\left\Vert h\right\Vert _{C_{1-\zeta\,; \,\Psi }\left[ a,x\right] }=\underset{t\in \left[ a,x\right] }{\max }\left\vert \left( \Psi \left( t\right) -\Psi \left( a\right) \right)
 ^{1-\zeta }h\left( t\right) \right\vert.
 \end{equation}
 
By mathematical induction, we now prove , for each $x\in [a,a+\chi]$ and $y_j\in C_{1-\zeta\,; \,\Psi}[a,x]$, 
\begin{equation}\label{e10}
\left\Vert y_{n+1}-y_{n}  \right\Vert_{_{C_{1-\zeta\,; \,\Psi }\left[ a,x\right] }}\leq \frac{ M \Gamma(\zeta)} {L}\frac{\left[L(\Psi(x)-\Psi(a))^\eta\right]^{n+1}}{\Gamma((n+1)\eta+\zeta)}, ~n\in \N.
 \end{equation}
 
Using Lemma \ref{lema3}, we obtain
 \begin{align*}
 &\left\Vert y_{1}-y_{0}\right\Vert_{_{C_{1-\zeta\,; \,\Psi }\left[ a,x\right] }}\\
 &=\underset{t \in [a,x]}{\max}\left|(\Psi (t)-\Psi (a) )^{1-\zeta }\left\{ y_{0}(t)+\frac{1}{\Gamma ( \eta ) }\int_{a}^{t}\mathrm{Q}^{\eta}_{\Psi}(t,\sigma)f(\sigma,y_{0}(\sigma)) \,d\sigma-y_{0}(t)\right\}\right|\\
 &=\underset{t \in [a,x]}{\max}\left|(\Psi (t)-\Psi (a) )^{1-\zeta }\frac{1}{\Gamma ( \eta ) }\int_{a}^{t}\mathrm{Q}^{\eta}_{\Psi}(t,\sigma)f(\sigma,y_{0}(\sigma)) \,d\sigma\right|\\
 &=\left\Vert \mathcal{I}_{a+}^{\eta\,; \,\Psi}f\left(\cdot,y_{0}(\cdot)\right)\right\Vert_{_{C_{1-\zeta\,; \,\Psi }\left[ a,x\right] }}\\
 &\leq M\,\,\frac{\Gamma(\zeta)}{\Gamma(\eta+\zeta)}\left(\Psi(t)-\Psi(a)\right)^{\eta}\\
 &\leq M\,\,\frac{\Gamma(\zeta)}{\Gamma(\eta+\zeta)}\left(\Psi(x)-\Psi(a)\right)^{\eta},
 \end{align*}  
which is the inequality \eqref{e10} for $n=0$. 

Now, assume that the inequality \eqref{e10} is hold for $n=k$. We prove  it is also hold for $n=k+1$. In fact, 
 \begin{align*}
 &\left\Vert y_{k+2}-y_{k+1}\right\Vert_{C_{1-\zeta\,; \,\Psi }\left[ a,x\right] }\\
&\leq L\underset{t \in [a,x]}{\max}(\Psi (t)-\Psi (a) )^{1-\zeta }
\frac{1}{\Gamma ( \eta ) }\int_{a}^{t}\mathrm{Q}^{\eta}_{\Psi}(t,\sigma)\times \\
&\qquad( \Psi (\sigma) -\Psi (a)) ^{\zeta-1}( \Psi (\sigma) -\Psi (a)) ^{1-\zeta}\left| y_{k+1}(\sigma)-y_{k}(\sigma)\right| \,d\sigma\\
&\leq L(\Psi (x)-\Psi (a) )^{1-\zeta }\frac{1}{\Gamma ( \eta ) }\int_{a}^{x}\mathrm{Q}^{\eta:\Psi}(x,\sigma)( \Psi (\sigma) -\Psi (a)) ^{\zeta-1}\left\Vert y_{k+1}-y_{k}\right\Vert_{{C_{1-\zeta\,; \,\Psi }\left[ a,\sigma\right] }}\,d\sigma\\
&\leq L(\Psi (x)-\Psi (a) )^{1-\zeta }\frac{1}{\Gamma ( \eta ) }\int_{a}^{x}\mathrm{Q}^{\eta:\Psi}(x,\sigma)( \Psi (\sigma) -\Psi (a)) ^{\zeta-1}\times \\&\qquad\frac{ M \Gamma(\zeta)} {L}\frac{\left[L(\Psi(\sigma)-\Psi(a))^\eta\right]^{k+1}}{\Gamma((k+1)\eta+\zeta)}\,d\sigma\\
&=M\,\,L^{k+1}\,\,\frac{\Gamma(\zeta)}{\Gamma((k+1)\eta+\zeta)}(\Psi(x)-\Psi(a))^{1-\zeta}\,\,\mathcal{I}_{a+}^{\eta\,; \,\Psi}( \Psi (x) -\Psi (a)) ^{(k+1)\eta+\zeta-1}\\
&=M\,\,L^{k+1}\,\,\frac{\Gamma(\zeta)}{\Gamma((k+1)\eta+\zeta)}(\Psi(x)-\Psi(a))^{1-\zeta}\times\frac{\Gamma((k+1)\eta+\zeta)}{\Gamma((k+2)\eta+\zeta)}(\Psi(x)-\Psi(a))^{(k+2)\eta+\zeta-1}\\
&=\frac{M\,\,\Gamma(\zeta)}{L}\frac{\left( L\left(\Psi(x)-\Psi(a)\right)^{\eta}\right) ^{k+2}}{\Gamma((k+2)\eta+\zeta)}.
\end{align*}
Thus, inequality \eqref{e10} is true for $n=k+1$. Hence, by using principle of mathematical induction the inequality  \eqref{e10} holds for each $n\in \N$ and  $x \in [a,a+\chi]$. Taking $x \to a+\chi$ in \eqref{e10}, we get, 
$$
\left\Vert y_{k}-y_{k-1}\right\Vert_{C_{1-\zeta\,;\,\Psi}[a,a+\chi]}\leq\frac{ M \Gamma(\zeta)} {L}\frac{\left[L(\Psi(a+\chi)-\Psi(a))^\eta\right]^{k}}{\Gamma(k\eta+\zeta)}.
$$

Therefore,
\begin{align*}
\sum_{k=1}^{\infty}\left\Vert y_{k}-y_{k-1}\right\Vert_{C_{1-\zeta\,;\,\Psi}[a,a+\chi]}
&\leq\frac{ M \Gamma(\zeta)} {L}\sum_{k=1}^{\infty}\frac{\left[L(\Psi(a+\chi)-\Psi(a))^\eta\right]^{k}}{\Gamma(k\eta+\zeta)}\\
&\leq\frac{ M \Gamma(\zeta)} {L}\left[\mathcal{E}_{\eta,\,\zeta}\left(L\left(\Psi(a+\chi)-\Psi(a)\right)^\eta\right)-\frac{1}{\Gamma(\zeta)}\right].
\end{align*}

This proves the series
$y_0+\sum_{k=1}^{\infty}\left(y_k - y_{k-1}\right)$ is convergent in the space $C_{1-\zeta\,; \,\Psi}[a,a+\chi]$.  Let us suppose that $$\tilde{y}=y_0+\sum_{k=1}^{\infty}(y_k - y_{k-1}).$$

Therefore,
 \begin{equation}\label{e11}
\left\Vert y_{n}-\tilde{y}\right\Vert_{C_{1-\zeta\,; \,\Psi}[a,a+\chi]}\rightarrow 0~ \text{as}\,\,n\rightarrow \infty.
 \end{equation}

\noindent 

{{Step 3}:} $\tilde{y}$ is solution of fractional integral equation \eqref{eq3}. In fact, we have
\begin{align*}
&\left\Vert\mathcal{I}_{a+}^{\eta\,; \,\Psi}f\left(\cdot,y_{n}(\cdot)\right)-\mathcal{I}_{a+}^{\eta\,; \,\Psi}f\left(\cdot,\tilde{y}(\cdot)\right)\right\Vert_{C_{1-\zeta\,;\,\Psi}[a,a+\chi]}\\
&\leq L\left\Vert         y_{n}-\tilde{y}\right\Vert_{C_{1-\zeta\,;\,\Psi}[a,a+\chi]}\underset{t \in [a,a+\chi]}{\max}(\Psi (t)-\Psi (a) )^{1-\zeta }\,\,\mathcal{I}_{a+}^{\eta\,; \,\Psi}( \Psi (t) -\Psi (a)) ^{\zeta-1}\\
&\leq L\left\Vert         y_{n}-\tilde{y}\right\Vert_{C_{1-\zeta\,;\,\Psi}[a,a+\chi]}\underset{t \in [a,a+\chi]}{\max}(\Psi (t)-\Psi (a) )^{1-\zeta}\frac{\Gamma(\zeta)}{\Gamma(\eta+\zeta)}(\Psi (t)-\Psi (a) )^{\eta+\zeta-1}\\
&\leq L\,\,\ \frac{\Gamma(\zeta)}{\Gamma(\eta+\zeta)}(\Psi (a+\chi)-\Psi (a) )^{\eta}
\left\Vert y_{n}-\tilde{y}\right\Vert_{C_{1-\zeta\,;\,\Psi}[a,a+\chi]}.
\end{align*}

Using \eqref{e11}, we  obtain
\begin{equation}\label{e12}
\left\Vert\mathcal{I}_{a+}^{\eta\,; \,\Psi}f\left(\cdot,y_{n}(\cdot)\right)-\mathcal{I}_{a+}^{\eta\,; \,\Psi}f\left(\cdot,\tilde{y}(\cdot)\right)\right\Vert_{C_{1-\zeta\,;\,\Psi}[a, a+\chi]}\rightarrow 0\,\,\, \text{as}\,\, n\rightarrow \infty.
\end{equation} 

Therefore, taking limit as $n \rightarrow \infty$ in \eqref{e7}, we get
\begin{align*}
\tilde{y}(t)&=y_0(t)+\mathcal{I}_{a+}^{\eta\,; \,\Psi}f\left(t,\tilde{y}(t)\right)\\
&=\mathcal{H}^{\Psi}_{\zeta}(t,a)\,y_a+\frac{1}{\Gamma ( \eta ) }\int_{a}^{t}\mathrm{Q}^{\eta}_{\Psi}(t,\sigma)f(\sigma,\tilde{y}(\sigma)) \,d\sigma . 
\end{align*}

This proves, $\tilde{y}\in C_{1-\zeta\,; \,\Psi}[a, a+\chi]$ is solution of \eqref{eq3}.
\noindent

{{Step 4}:} Uniqueness of the solution.\\
Let $y(t)$\,\,and\,\,\,${y}^\ast(t)$ be any two solutions of the  IVP \eqref{eq1}--\eqref{eq2} and consider the function defined by $z(t)=\left|y(t)-{y}^\ast(t)\right| $. Then, we get
\begin{align*}
z(t)&=\left|y(t)-{y}^\ast(t)\right|\\
&=\left|\frac{1}{\Gamma ( \eta ) }\int_{a}^{t}\mathrm{Q}^{\eta}_{\Psi}(t,\sigma)\left[f(\sigma,y(\sigma))-f(\sigma,{y}^\ast(\sigma))\right] \,d\sigma\right|\\
&\leq\frac{1}{\Gamma ( \eta ) }\int_{a}^{t}\mathrm{Q}^{\eta}_{\Psi}(t,\sigma)\left| f(\sigma,y(\sigma))-f(\sigma,{y}^\ast(\sigma)) \right|\,d\sigma\\
&\leq \frac{L}{\Gamma ( \eta ) }\int_{a}^{t}\mathrm{Q}^{\eta}_{\Psi}(t,\sigma)\left| y(\sigma)-{y}^\ast(\sigma) \right|\,d\sigma\\
&\leq \frac{L}{\Gamma ( \eta ) }\int_{a}^{t}\mathrm{Q}^{\eta}_{\Psi}(t,\sigma)z(\sigma)\,d\sigma.
\end{align*}
Using Gronwall inequality given in  the Theorem \ref{lema4}, we get
$$ z(t)\leq 0\times \mathcal{E}_{\eta}[L \left(\Psi(t)-\Psi(a)\right)^{\eta}].$$

Therefore $z(t)=0$ and we have $ y(t)={y}^\ast(t)$.
\end{proof}

\begin{rem}
For  $\eta=1,\, \nu=1$ and $\Psi(t)=t$ the above theorems includes the results  of Coddington {\rm[\cite{Coddington},Chapter 5]} for the ordinary differential equations. 
\end{rem}

\begin{theorem}	
Let $\{y_n\}$ be the sequence of Picard's successive approximation defined by \eqref{e7} and $y$ is the solution of the IVP \eqref{eq1}-\eqref{eq2}. Then the error $ y -y_{n}$ satisfies the condition 
\begin{equation}\label{eqn}
\| y -y_{n}\|_{C_{1-\zeta\,;\,\Psi}[a,a+\chi]}\leq\frac{M\Gamma(\zeta)}{L}\left( \mathcal{E}_{\eta ,\, \zeta}(L(\Psi(a+\chi)-\Psi(a))^{\eta})-\sum_{k=0}^{n}\frac{(L(\Psi(a+\chi)-\Psi(a))^{\eta})^{k}}{\Gamma(k\eta+\zeta)} \right). 
\end{equation}

\end{theorem}
\begin{proof}
From the proof of the Theorem \ref {theorem4.2},  we have
$$y=y_0+\sum_{k=1}^{\infty}(y_{k} - y_{k-1})$$
and
$$y_{n}=y_0+\sum_{k=1}^{n}(y_{k} - y_{k-1}).$$
Hence
$$ y -y_{n}=\sum_{k=n+1}^{\infty}(y_{k} - y_{k-1}).$$

From the above relations, in view of \eqref{e10}, we have 
\begin{align*}
&\| y -y_{n}\|_{C_{1-\zeta\,;\,\Psi}[a,a+\chi]}\\
&\leq \sum_{k=n+1}^{\infty}\|y_{k}-y_{k-1}\|_{C_{1-\zeta\,;\,\Psi}[a,a+\chi]}\\
&\leq \sum_{k=n+1}^{\infty} \frac{M \Gamma(\zeta)}{L}\frac{\left( L (\Psi(a+\chi)-\Psi(a))^{\eta}\right)^{k }}{\Gamma(k\eta+\zeta)}\\
&\leq\frac{M \Gamma(\zeta)}{L}\left(\sum_{k=0}^{\infty} \frac{\left(L(\Psi(a+\chi)-\Psi(a))^{\eta}\right)^{k}}{\Gamma(k\eta+\zeta)}-\sum_{k=0}^{n} \frac{\left(L(\Psi(a+\chi)-\Psi(a))^{\eta}\right)^{k }}{\Gamma(k\eta+\zeta)}\right)\\
&\leq \frac{M\Gamma(\zeta)}{L}\left( \mathcal{E}_{\eta ,\, \zeta}(L(\Psi(a+\chi)-\Psi(a))^{\eta})-\sum_{k=0}^{n}\frac{(L(\Psi(a+\chi)-\Psi(a))^{\eta})^{k}}{\Gamma(k\eta+\zeta)} \right).
\end{align*}
\end{proof}
\begin{rem}

From the inequality \eqref{eqn}, it follows that 
\begin{small}
\begin{align*}
&\lim\limits_{n\rightarrow \infty} \| y -y_{n}\|_{C_{1-\zeta\,;\,\Psi}[a,a+\chi]}\\
&\leq \frac{M\Gamma(\zeta)}{L}\lim\limits_{n\rightarrow \infty}\left(\sum_{k=0}^{\infty} \frac{\left(L(\Psi(a+\chi)-\Psi(a))^{\eta}\right)^{k}}{\Gamma(k\eta+\zeta)}-\sum_{k=0}^{n} \frac{\left(L(\Psi(a+\chi)-\Psi(a))^{\eta}\right)^{k }}{\Gamma(k\eta+\zeta)}\right) \\
&= \frac{M\Gamma(\zeta)}{L}\left(\mathcal{E}_{\eta ,\, \zeta}(L(\Psi(a+\chi)-\Psi(a))^{\eta})-\sum_{k=0}^{\infty}\frac{\left(L(\Psi(a+\chi)-\Psi(a))^{\eta}\right)^{k }}{\Gamma(k\eta+\zeta)}\right)\\
&=  \frac{M\Gamma(\zeta)}{L}\left(\mathcal{E}_{\eta ,\, \zeta}(L(\Psi(a+\chi)-\Psi(a))^{\eta})-\mathcal{E}_{\eta,\,\zeta}(L(\Psi(a+\chi)-\Psi(a))^{\eta}\right)\\
&=0.
\end{align*}
This implies the sequence $\{y_n \}$ of successive approximation converges to the solution $y $ of the problem \eqref{eq1}-\eqref{eq2} as $ n \rightarrow \infty.$
\end{small}
\end{rem}

\section{Picard's Successive Approximations: Linear Case}
Here we derive the representation formula  for the solution of linear Cauchy problem with constant coefficient and variable coefficients which extend the results of \cite[Chapter 7]{Gorenflo}.

\subsection{Linear Cauchy Type problem with Constant Coefficient}

\begin{theorem} \label{th6.1}
Let  $f\in C_{1-\zeta\,; \,\Psi}\left( \Delta,\,\R\right) ]$ and $\lambda\in \R$. Then, the solution of  the Cauchy problem for FDE with constant coefficient involving $ \Psi $-Hilfer fractional derivative
\eqref{a1}-\eqref{a2}
is given by
\begin{align*}
y\left( t\right)&=y_a(\Psi(t)-\Psi(a))^{\zeta-1} \,\mathcal{E}_{\eta,\,\zeta}\left( \lambda(\Psi(t)-\Psi(a))^{\eta}\right) \\
&\qquad+\int_{a}^{t}\mathrm{Q}^{\eta}_{\Psi}(t,\sigma)\,\mathcal{E}_{\eta,\,\eta}\left( \lambda(\Psi(t)-\Psi(\sigma))^{\eta}\right) \,f(\sigma)\,\,d\sigma\notag
 \end{align*}
 where,~$\mathrm{Q}^{\eta}_{\Psi}(t,\sigma)=\Psi'(\sigma)(\Psi(t)-\Psi(\sigma))^{\eta-1} $
\end{theorem}
\begin{proof} The linear Cauchy problem \eqref{a1}-\eqref{a2} is equivalent to  
\begin{align}\label{e63}
y (t) &=\mathcal{H}^{\Psi}_{\zeta}(t,a)\,y_{a}+\frac{\lambda}{
\Gamma ( \eta ) }\int_{a}^{t}\mathrm{Q}^{\eta}_{\Psi}(t,\sigma)y(\sigma)\,d\sigma+\frac{1}{\Gamma ( \eta ) }\int_{a}^{t}\mathrm{Q}^{\eta}_{\Psi}(t,\sigma)f(\sigma) \,d\sigma.
\end{align}
Employing the method of successive approximation we derive the solution of  \eqref{e63}. Consider the sequences $\{y_{m}\}$ defined by
 \begin{align*}
 y_0(t)&=\mathcal{H}^{\Psi}_{\zeta}(t,a)\,y_{a}\\
 y_m(t)&=y_0(t)+\frac{\lambda}{ \Gamma ( \eta ) }\int_{a}^{t}\mathrm{Q}^{\eta}_{\Psi}(t,\sigma)y_{m-1}(\sigma) \,d\sigma\\
 &\qquad+\frac{1}{ \Gamma ( \eta ) }\int_{a}^{t}\mathrm{Q}^{\eta}_{\Psi}(t,\sigma)f(\sigma) \,d\sigma, ~(m=1,2,\cdots).
 \end{align*}

By method of induction, we prove that
 \begin{small}
 \begin{align}\label{e18}
 y_m(t)=y_{a}\sum_{k=0}^{m}\frac{\lambda^k( \Psi (t) -\Psi (a)) ^{k\eta+\zeta-1}}{\Gamma ( k\eta+\zeta )}+\int_{a}^{t}\Psi ^{\prime }(\sigma)\sum_{k=0}^{m-1}\frac{\lambda^{k}( \Psi (t) -\Psi (\sigma)) ^{(k+1)\eta-1}}{\Gamma ((k+1)\eta )} f(\sigma)\,d\sigma.
 \end{align}
 \end{small} 

For the case $m=1$, we have
 \begin{align*}
y_1(t)&=y_0(t)+\frac{\lambda}{
 \Gamma ( \eta ) }\int_{a}^{t}\mathrm{Q}^{\eta}_{\Psi}(t,\sigma)y_{0}(\sigma)\,d\sigma+\frac{1}{\Gamma ( \eta ) }\int_{a}^{t}\mathrm{Q}^{\eta}_{\Psi}(t,\sigma)f(\sigma) \,d\sigma\\
 &=y_0(t)+\frac{\lambda}{\Gamma ( \eta ) }\int_{a}^{t}\mathrm{Q}^{\eta}_{\Psi}(t,\sigma)\frac{( \Psi (\sigma)-\Psi (a) ) ^{\zeta -1}}{\Gamma ( \zeta ) }y_{a}\, \,d\sigma\\
 &\qquad+\frac{1}{\Gamma ( \eta ) }\int_{a}^{t}\mathrm{Q}^{\eta}_{\Psi}(t,\sigma)f(\sigma) \,d\sigma\\
 &=y_0(t)+y_{a}\frac{\lambda}{\Gamma ( \zeta ) }\, \mathcal{I}_{a+}^{\eta\,; \,\Psi}(\Psi(t)-\Psi(a))^{\zeta-1}+
 \mathcal{I}_{a+}^{\eta\,; \,\Psi}f(t)\\
 &=y_{a}\,\frac{\left( \Psi (t)-\Psi (a) \right) ^{\zeta -1}}{\Gamma ( \zeta ) }+y_{a} \, \frac{\lambda}{
 \Gamma ( \eta +\zeta) } ( \Psi (t) -\Psi (a)) ^{\eta+\zeta-1}\\
 &\qquad+\frac{1}{\Gamma ( \eta ) }\int_{a}^{t}\mathrm{Q}^{\eta}_{\Psi}(t,\sigma)f(\sigma) \,d\sigma.
  \end{align*}

This is the equation \eqref{e18} for $m=1$. Now, we assume that the equation \eqref{e18} is hold for $m=j$ and prove that it is
 also true for  $m=j+1$. In fact, 
 \begin{small}  
 \begin{align*}
 y_{j+1}(t) 
 &=y_0(t)+\frac{\lambda}{
 \Gamma ( \eta ) }\int_{a}^{t}\mathrm{Q}^{\eta}_{\Psi}(t,\sigma)y_{j}(\sigma) \,d\sigma+\frac{1}{\Gamma ( \eta ) }\int_{a}^{t}\mathrm{Q}^{\eta}_{\Psi}(t,\sigma)f(\sigma) \,d\sigma\\
 &=y_0(t)+\frac{\lambda}{\Gamma ( \eta ) }\int_{a}^{t}\mathrm{Q}^{\eta}_{\Psi}(t,\sigma) \left\{ y_{a}\sum_{k=0}^{j}\frac{\lambda^k ( \Psi (\sigma)-\Psi (a) ) ^{k\eta+\zeta -1}}{\Gamma (k\eta+ \zeta ) }\right.\\
  &\left.\qquad+\int_{a}^{\sigma}\Psi ^{\prime }(\tau)\sum_{k=0}^{j-1}
 \frac{\lambda^{k}}{\Gamma((k+1) \eta)} ( \Psi (\sigma)-\Psi(\tau))^{(k+1)\eta-1} f(\tau)d\tau)  \right \}\,d\sigma+\mathcal{I}_{a+}^{\eta\,; \,\Psi}f(t)\\
 &=y_0(t)+y_{a}\sum_{k=0}^{j}\frac{ \lambda^{k+1} }{\Gamma (k\eta+ \zeta ) }\,\mathcal{I}_{a+}^{\eta\,; \,\Psi}(\Psi(t)-\Psi(a))^{k\eta+\zeta-1}\\
 &\qquad +\sum_{k=0}^{j-1}\frac{ \lambda^{k+1} }{\Gamma ((k+1)\eta )}
  \frac{1}{\Gamma ( \eta )}\int_{a}^{t}\mathrm{Q}^{\eta}_{\Psi}(t,\sigma) \times\\
 &\qquad \qquad \left\{\int_{a}^{\sigma}\Psi ^{\prime }(\tau)( \Psi (\sigma)-\Psi(\tau))^{(k+1)\eta-1} f(\tau)d\tau \right \}\,d\sigma+\mathcal{I}_{a+}^{\eta\,; \,\Psi}f(t).
 \end{align*}
 \end{small}
 Changing the order of integration in the second last term, 
 \begin{small}  
  \begin{align*}
 & y_{j+1}(t) \\
 &=y_{a}\frac{\left( \Psi (t)-\Psi (a) \right) ^{\zeta -1}}{\Gamma ( \zeta ) }+y_{a}\sum_{k=0}^{j}\frac{ \lambda^{k+1} }{\Gamma ((k+1)\eta+\zeta)}(\Psi(t)-\Psi(a))^{(k+1)\eta+\zeta-1}\\
 &\qquad +\sum_{k=0}^{j-1}\frac{ \lambda^{k+1} }{\Gamma ((k+1)\eta )}\frac{1}{\Gamma ( \eta )}\int_{a}^{t}\Psi ^{\prime }(\tau) \left\{\int_{\tau}^{t}\Psi ^{\prime }(\sigma)( \Psi (\sigma)-\Psi(\tau))^{(k+1)\eta-1}(\Psi (t) -\Psi (\sigma)) ^{\eta-1} f(\sigma)\,d\sigma \right \} f(\tau)d\tau\\
 &\qquad+\frac{1}{\Gamma ( \eta ) }\int_{a}^{t}\mathrm{Q}^{\eta}_{\Psi}(t,\sigma)f(\sigma) \,d\sigma\\
 &=y_{a}\sum_{k=0}^{j+1}\frac{ \lambda^{k} }{\Gamma (k\eta+ \zeta)}(\Psi(t)-\Psi(a))^{k\eta+\zeta-1}\\
 &\qquad +\sum_{k=0}^{j-1}\frac{ \lambda^{k+1} }{\Gamma ((k+1)\eta )}\frac{1}{\Gamma ( \eta )}\int_{a}^{t}\frac{\Gamma(\eta)\Gamma((k+1)\eta)}{\Gamma((k+2)\eta)}( \Psi (t)-\Psi(\tau))^{(k+2)\eta-1} \Psi ^{\prime }(\tau) f(\tau)d\tau\\
 &\qquad +\frac{1}{\Gamma ( \eta ) }\int_{a}^{t}\mathrm{Q}^{\eta}_{\Psi}(t,\sigma)f(\sigma) \,d\sigma\\
 &=y_{a}\sum_{k=0}^{j+1}\frac{ \lambda^{k} }{\Gamma (k\eta+ \zeta)}(\Psi(t)-\Psi(a))^{k\eta+\zeta-1}\\ &\qquad+\int_{a}^{t}\left\{\sum_{k=0}^{j-1}\frac{ \lambda^{k+1} }{\Gamma((k+2)\eta)}(\Psi(t)-\Psi(\sigma))^{(k+2)\eta-1}
 +\frac{( \Psi (t) -\Psi (\sigma)) ^{\eta-1}}{\Gamma(\eta)}\right\} \Psi ^{\prime }(\sigma) f(\sigma) \,d\sigma\\
 &=y_{a}\sum_{k=0}^{j+1}\frac{ \lambda^{k} }{\Gamma (k\eta+ \zeta)}(\Psi(t)-\Psi(a))^{k\eta+\zeta-1}+\int_{a}^{t}\Psi ^{\prime }(\sigma)\sum_{k=0}^{j}\frac{ \lambda^{k} }{\Gamma((k+1)\eta)}(\Psi(t)-\Psi(\sigma))^{(k+1)\eta-1} f(\sigma) \,d\sigma.
 \end{align*}
\end{small}
This proves equation \eqref{e18} is true for $m=j+1$. By mathematical induction the equation \eqref{e18} is true for all $m\in\N$.
 Taking limit as $m\rightarrow \infty$, on both sides of \eqref{e18}, we get the solution of the Cauchy problem \eqref{a1}-\eqref{a2} as
\begin{align*}
y(t)&=\lim\limits_{ m\rightarrow \infty}y_m(t)\\
&=y_a\sum_{k=0}^{\infty}\frac{\lambda^k( \Psi (t) -\Psi (a)) ^{k\eta+\zeta-1}}{\Gamma ( k\eta+\zeta )}+\int_{a}^{t}\Psi ^{\prime }(\sigma)\sum_{k=0}^{\infty}\frac{\lambda^{k}( \Psi (t) -\Psi (\sigma)) ^{(k+1)\eta-1}}{\Gamma ((k+1)\eta )} f(\sigma)\,d\sigma\\
&=y_a( \Psi (t) -\Psi (a)) ^{\zeta-1}\mathcal{E}_{\eta,\,\zeta}\left( \lambda(\Psi(t)-\Psi(a))^{\eta}\right) \\
&\qquad+\int_{a}^{t}\mathrm{Q}^{\eta}_{\Psi}(t,\sigma) \mathcal{E}_{\eta,\,\eta}\left( \lambda(\Psi(t)-\Psi(\sigma))^{\eta}\right) f(\sigma)\,d\sigma.
\end{align*}
\end{proof}
\subsection{Linear Cauchy Type problem with Variable Coefficient}
\begin{theorem} \label{th7.1} Let $f\in C_{1-\zeta\,; \,\Psi}\left( \Delta,\R\right) ,~\lambda \in \R~\text{and}~\mu>1-\eta $. Then, the solution of Cauchy problem for homogeneous FDE with variable coefficient involving $ \Psi$-Hilfer fractional derivative
\eqref{a3}-\eqref{a4}
 is given by
  \begin{align*}
  y\left( t\right) =\mathcal{H}^{\Psi}_{\zeta}(t,a)\,y_{a}\, \mathcal{E}_{\eta,\,1+\frac{\mu-1}{\eta},\,\frac{\mu+\zeta-2}{2}} \left( \lambda(\Psi(t)-\Psi(a))^{\eta+\mu-1}\right).
   \end{align*}
   where, $\mathcal{H}^{\Psi}_{\zeta}(t,a)=\frac{\left( \Psi (t)-\Psi (a) \right) ^{\zeta-1}}{\Gamma(\zeta)}$
  \end{theorem}
\begin{proof} The equivalent fractional integral of  \eqref{a3}-\eqref{a4} is 
\begin{equation}\label{e21}
y(t)=\mathcal{H}^{\Psi}_{\zeta}(t,a)\,y_{a}+\frac{\lambda}{
\Gamma ( \eta ) }\int_{a}^{t}\mathrm{Q}^{\eta}_{\Psi}(t,\sigma)( \Psi (\sigma) -\Psi (a)) ^{\mu-1}y(\sigma) \,d\sigma.
\end{equation}
We consider the sequences $\{y_{m}\}$ of successive approximation defined by
\begin{align*}
y_0(t)&=\mathcal{H}^{\Psi}_{\zeta}(t,a)\,y_{a}\\
 y_m(t)&=y_0(t)+\frac{\lambda}{\Gamma ( \eta ) }\int_{a}^{t}\mathrm{Q}^{\eta}_{\Psi}(t,\sigma)( \Psi (\sigma) -\Psi (a)) ^{\mu-1}y_{m-1}(\sigma) \,d\sigma, ~(m=1,2,\cdots),
\end{align*}
and derive the solution of \eqref{e21}. By mathematical induction, we prove that
\begin{align}\label{e22}
 y_{m}(t)=\mathcal{H}^{\Psi}_{\zeta}(t,a)\,y_{a}\left\{1+\sum_{k=1}^{m}c_k\,\left( \lambda( \Psi (t) -\Psi (a)) ^{\eta+\mu-1}\right) ^k \right \},
 \end{align}
 where 
 \begin{align*}
c_k&=\prod_{j=0}^{k-1}\frac{\Gamma(j(\eta+\mu-1)+\mu+\zeta-1)}{\Gamma(j(\eta+\mu-1)+\eta+\mu+\zeta-1)}\\
&=\prod_{j=0}^{k-1}\frac{\Gamma(\eta[j(1+\frac{\mu-1}{\eta})+\frac{\mu+\zeta-2}{\eta}]+1)}{\Gamma(\eta[j(1+\frac{\mu-1}{\eta})+\frac{\mu+\zeta-2}{\eta}+1]+1)}.
 \end{align*}

For the case $m=1$, we have
 \begin{align*}
 y_1(t)
 &=y_0(t)+\frac{\lambda}{\Gamma ( \eta ) }\int_{a}^{t}\mathrm{Q}^{\eta}_{\Psi}(t,\sigma)( \Psi (\sigma) -\Psi (a)) ^{\mu-1}y_0(\sigma) \,d\sigma\\
&=\mathcal{H}^{\Psi}_{\zeta}(t,a)\,y_{a}+\frac{\lambda}{\Gamma ( \eta ) }\int_{a}^{t}\mathrm{Q}^{\eta}_{\Psi}(t,\sigma)( \Psi (\sigma) -\Psi (a)) ^{\mu-1}y_{a}\frac{( \Psi (\sigma)-\Psi (a) ) ^{\zeta -1}}{\Gamma ( \zeta ) } \,d\sigma\\
&=\mathcal{H}^{\Psi}_{\zeta}(t,a)\,y_{a}+y_{a}\frac{\lambda}{\Gamma ( \zeta ) }\,\, \mathcal{I}_{a+}^{\eta\,; \,\Psi}(\Psi(t)-\Psi(a))^{\mu+\zeta-2}\\
&=y_{a}\frac{( \Psi (t)-\Psi (a) ) ^{\zeta -1}}{\Gamma ( \zeta ) }+y_{a}\,\,\frac{\lambda}{\Gamma(\zeta)}\frac{\Gamma(\mu+\zeta-1)}{\Gamma(\eta+\mu+\zeta-1)}(\Psi(t)-\Psi(a))^{\eta+\mu+\zeta-2}\\
&=\mathcal{H}^{\Psi}_{\zeta}(t,a)\,y_{a}\left\{1+\frac{\Gamma(\mu+\zeta-1)}{\Gamma(\eta+\mu+\zeta-1)}\left( \lambda( \Psi (t) -\Psi (a)) ^{\eta+\mu-1}\right) \right \}.
\end{align*}

This is the equation \eqref{e22} for  $m=1$. Now, we assume that the equation \eqref{e22} is hold for $m=j$. we prove it is
 also hold for  $m=j+1$. In fact, 
 \begin{align*}
& y_{j+1}(t)\\
&= y_0(t)+\frac{\lambda}{\Gamma ( \eta ) }\int_{a}^{t}\mathrm{Q}^{\eta}_{\Psi}(t,\sigma)( \Psi (\sigma) -\Psi (a)) ^{\mu-1}y_{j}(\sigma) \,d\sigma\\
&= y_0(t)+\frac{\lambda}{\Gamma ( \eta ) }\int_{a}^{t}\mathrm{Q}^{\eta}_{\Psi}(t,\sigma)( \Psi (\sigma) -\Psi (a)) ^{\mu-1} \times\\
& \qquad \frac{( \Psi (\sigma)-\Psi (a) ) ^{\zeta -1}}{\Gamma (\zeta)}y_{a}\,\left\{1+\sum_{k=1}^{j}c_{k}\left( \lambda( \Psi (\sigma) -\Psi (a)) ^{\eta+\mu-1}\right) ^{k} \right \} \,d\sigma\\
&= y_0(t)+y_{a}\frac{\lambda}{\Gamma (\zeta)}\frac{1}{\Gamma ( \eta )} \int_{a}^{t}\mathrm{Q}^{\eta}_{\Psi}(t,\sigma)( \Psi (\sigma) -\Psi (a)) ^{\mu+\zeta-2} \,d\sigma\\
&\qquad+y_a\sum_{k=1}^{j}\frac{c_{k}\,\,\lambda^{k+1}}{\Gamma (\zeta)}\frac{1}{\Gamma ( \eta )}\int_{a}^{t}\mathrm{Q}^{\eta}_{\Psi}(t,\sigma)( \Psi (\sigma) -\Psi (a)) ^{(k+1)\mu+k\eta+\zeta-k-2} \,d\sigma\\
&=y_0(t)+y_{a}\frac{\lambda}{{\Gamma(\zeta)}}\,\,\mathcal{I}_{a+}^{\eta\,; \,\Psi}(\Psi(t)-\Psi(a))^{\mu+\zeta-2}\\
&\qquad+y_a\sum_{k=1}^{j}\frac{c_{k}\,\,\lambda^{k+1}}{\Gamma (\zeta)}\,\,\mathcal{I}_{a+}^{\eta\,; \,\Psi}(\Psi(t)-\Psi(a))^{(k+1)\mu+k\eta+\zeta-k-2}\\
&=y_{a}\frac{(\Psi(t)-\Psi(a))^{\zeta-1}}{\Gamma(\zeta)}+y_a\frac{\lambda \,\Gamma(\mu+\zeta-1)}{\Gamma(\zeta)\Gamma(\eta+\mu+\zeta-1)}(\Psi(t)-\Psi(a))^{\eta+\mu+\zeta-2}\\
&\qquad+y_a\sum_{k=1}^{j}\frac{c_{k}\,\,\lambda^{k+1}}{\Gamma (\zeta)}\frac{\Gamma((k+1)\mu+k\eta+\zeta-(k+1))}{\Gamma((k+1)(\mu+\eta)+\zeta-(k+1))}(\Psi(t)-\Psi(a))^{(k+1)(\eta+\mu)+\zeta-k-2}\\
&=\mathcal{H}^{\Psi}_{\zeta}(t,a)\,y_{a}\left\{1+\frac{\Gamma(\mu+\zeta-1)}{\Gamma(\eta+\mu+\zeta-1)}\lambda(\Psi(t)-\Psi(a))^{\eta+\mu-1}\right.\\
&\qquad~\left.+\sum_{k=1}^{j}c_{k}\,\,\frac{\Gamma((k+1)\mu+k\eta+\zeta-(k+1)}{\Gamma((k+1)(\mu+\eta)+\zeta-(k+1)}\lambda^{k+1}(\Psi(t)-\Psi(a))^{(k+1)(\eta+\mu-1)}\right\}\\
&=\mathcal{H}^{\Psi}_{\zeta}(t,a)\,y_{a}\left\{1+\frac{\Gamma(\mu+\zeta-1)}{\Gamma(\eta+\mu+\zeta-1)}\lambda(\Psi(t)-\Psi(a))^{\eta+\mu-1}
\right.\\
&\qquad~+\left.\sum_{k=1}^{j}c_{k}\frac{\Gamma(k(\eta+\mu-1)+\mu+\zeta-1)}{\Gamma(k(\eta+\mu-1)+\eta+\mu+\zeta-1)}\lambda^{k+1}(\Psi(t)-\Psi(a))^{(k+1)(\eta+\mu-1)}\right\}\\
&=\mathcal{H}^{\Psi}_{\zeta}(t,a)\,y_{a}
\left\{1+\sum_{k=1}^{j+1}c_{k}{\left( \lambda(\Psi(t)-\Psi(a))^{(\eta+\mu-1)}\right) }^{k}\right\}
\end{align*}
where, $$c_k=\prod_{j=0}^{k-1}\frac{\Gamma(j(\eta+\mu-1)+\mu+\zeta-1)}{\Gamma(j(\eta+\mu-1)+\eta+\mu+\zeta-1)} .$$

This proves equation \eqref{e22} is true for $m=j+1$. By mathematical induction the equation \eqref{e22} is true for all $m\in\N$. Taking limit as $m\rightarrow \infty$, on both sides of \eqref{e22}, we get the solution of the Cauchy problem \eqref{a3}-\eqref{a4}, given by 
 \begin{align*}
 y(t)&=\lim\limits_{ m\rightarrow \infty}y_m(t)\\
 &=\mathcal{H}^{\Psi}_{\zeta}(t,a)\,y_{a}\left\{1+\sum_{k=1}^{\infty}c_k[\lambda( \Psi (t) -\Psi (a)) ^{\eta+\mu-1}]^k \right \}\\
 &=\mathcal{H}^{\Psi}_{\zeta}(t,a)\,y_{a} \,\mathcal{E}_{\eta,\,1+\frac{\mu-1}{\eta},\,\frac{\mu+\zeta-2}{2}}\left(  \lambda(\Psi(t)-\Psi(a))^{\eta+\mu-1}\right).
 \end{align*}
 where $\mathcal{E}_{\eta,\,1+\frac{\mu-1}{\eta},\,\frac{\mu+\zeta-2}{2}}(\cdot)$ is generalized (Kilbas--Saigo) Mittag--Leffler type function of three parameters.
 \end{proof}

\section{Concluding remarks}
In the present paper, we are able to provide a brief study of the theory of FDEs by means of the $\Psi$-Hilfer fractional derivative. We examined the existence along with the interval of existence, uniqueness, dependence of solutions and Picard's successive approximations in cases: nonlinear and linear.

The existence results pertaining to Cauchy  problem (\ref{eq1})-(\ref{eq2}), was obtained through the theorems of Schauder and Arzela-Ascoli. On the other hand, we can prove the uniqueness and continuous dependence of the problem (\ref{eq1})-(\ref{eq2}) by making use of mathematical induction and Weissinger's fixed point theorem. In addition, we investigate Picard's successive approximation in the nonlinear case for the Cauchy problem (\ref{eq1})-(\ref{eq2}) and in the linear cases: (\ref{a1})-(\ref{a2}) with constant coefficients and (\ref{a3})-(\ref{a4}) with variable coefficients, making use of the Gronwall's inequality and mathematical induction. It should be noted that the results obtained in the space of the weighted functions $C_{1-\zeta;\,\Psi}\left( \Delta,\R\right) $ are contributions to the fractional calculus field, in particular, the fractional analysis.

The question that arises is: will it be possible to perform the same study in space $L_{p}([0,1],\mathbb{R})$ with the norm $||(\cdot)||_{p,\,\eta}$, involving $\Psi$-Hilfer fractional derivative and Banach's fixed point theorem \cite{vanterler3}? If the answer is yes. What are the restrictions for such results, if any? These issues and others, are studies that are in progress and will be published in the future.


\end{document}